\documentclass[]{interact}

\usepackage{epstopdf}
\usepackage[caption=false]{subfig}

\usepackage[numbers,sort&compress]{natbib}
\bibpunct[, ]{[}{]}{,}{n}{,}{,}

\theoremstyle{plain}
\newtheorem{theorem}{Theorem}[section]
\newtheorem{lemma}[theorem]{Lemma}
\newtheorem{corollary}[theorem]{Corollary}
\newtheorem{proposition}[theorem]{Proposition}

\theoremstyle{definition}
\newtheorem{definition}[theorem]{Definition}

\newtheorem{assumption}[theorem]{Assumption}

\theoremstyle{remark}

\usepackage{amsfonts}       
\usepackage{nicefrac}       
\usepackage{microtype}      
\usepackage{times}
\usepackage{graphicx} 		
\usepackage{algorithm}
\usepackage{algorithmic}
\usepackage{amssymb}
\usepackage{amsmath}
\usepackage{bm}
\usepackage{color}
\usepackage{caption}
\usepackage{tikz}

\begin{document}

\articletype{ARTICLE TEMPLATE}

\title{Newton-type Multilevel Optimization Method}

\author{
\name{Chin Pang Ho\textsuperscript{a}\thanks{a. Email: clint.ho@cityu.edu.hk} and Michal Ko\v{c}vara\textsuperscript{b,c}\thanks{b. Email: m.kocvara@bham.ac.uk} and Panos Parpas\textsuperscript{d}\thanks{d. Email: p.parpas@imperial.ac.uk}}
\affil{\textsuperscript{a}School of Data Science, City University of Hong Kong, Hong Kong; \textsuperscript{b}School of Mathematics, The University of Birmingham, United Kingdom; \textsuperscript{c}Institute of Information Theory and Automation, Academy of Sciences of the Czech Republic, Prague, Czech Republic; \textsuperscript{d}Department of Computing, Imperial College London, United Kingdom}
}

\maketitle

\begin{abstract}
Inspired by multigrid methods for linear systems of equations, multilevel optimization methods have been proposed to solve structured optimization problems. Multilevel methods make more assumptions regarding the structure of the optimization model, and as a result, they outperform single-level methods, especially for large-scale models. The impressive performance of multilevel optimization methods is an empirical observation, and no theoretical explanation has so far been proposed. In order to address this issue, we study the convergence properties of a multilevel method that is motivated by second-order methods. We take the first step toward establishing how the structure of an optimization problem is related to the convergence rate of multilevel algorithms.
\end{abstract}


\section{Introduction}\label{sec:NeMOintro}

Multigrid methods are a well-known and established method for solving differential equations \cite{Briggs2000,Hackbusch1985,Han2017,Strang2007,Trottenberg2001,Wesseling1992}. When solving a differential equation using numerical methods, an approximation of the solution is obtained on a mesh via discretization. The computational cost of solving the discretized problem, however, varies and it depends on the choice of the mesh size used. Therefore, by considering different mesh sizes, a hierarchy of discretized models can be defined. In general, a more accurate solution can be obtained when a smaller mesh size is chosen, which results in a discretized problem in higher dimensions. We shall follow the traditional terminology in the multigrid literature and call a \textit{fine model} to be the discretization in which its solution is sufficiently close to the solution of the original differential equation; otherwise we call it a \textit{coarse model} \cite{Briggs2000}. The main idea of multigrid methods is to make use of the geometric similarity between different discretizations. In particular, during the iterative process of computing the solution of the fine model, one replaces part of the computations with the information from coarse models. The advantages of using multigrid methods are twofold. Firstly, coarse models are in lower dimensions compared to the fine model, and so the computational cost is reduced. Secondly and interestingly, the corrections generated by the coarse model and fine model are in fact complementary. It has been shown that using the fine model is effective in reducing the high frequency components of the residual (error) but ineffective in reducing the low frequency component of the error. Those low frequency components of the error, however, will become high frequency errors in the coarse model. Thus, they could be eliminated effectively using coarse models \cite{Briggs2000,Strang2007}.

This idea of multigrid was extended to optimization algorithms. Nash \cite{Nash2000} proposed a multigrid framework for unconstrained infinite-dimensional convex optimization problems. Examples of such problems could be found in the area of optimal control. Following the idea of Nash, many multigrid optimization methods were further developed \cite{Nash2000,Nash2014,Lewis2013,Lewis2005,Kocvara2016,Wen2009,Gratton2008}. In particular, Wen and Goldfarb \cite{Wen2009} provided a line search-based multigrid optimization algorithm under the framework in \cite{Nash2000}, and further extended the framework to nonconvex problems. Gratton et al. \cite{Gratton2008} provided a sophisticated trust-region version of multigrid optimization algorithms, which they called it multiscale algorithm. In this paper, we will consistently use the name \textit{multilevel algorithms} for all these optimization algorithms, but we emphasize that the terms multilevel, multigrid, and multiscale were used interchangeably in different papers. On the other hand, we keep the name \textit{multigrid methods} for the conventional multigrid methods that solve linear or nonlinear equations that are discretizations arising from partial differential equations (PDEs).

It is worth mentioning that different multilevel algorithms were developed beyond infinite-dimensional problems, such as Markov decision processes \cite{Ho2014}, semidefinite programming \cite{MR3742696}, artificial neural networks \cite{cal2019approximation}, and composite optimization for both the convex \cite{MR3572365} and non-convex case \cite{MR3716579}. Also, Calandra et al. \cite{cal2019highorder} proposed a multilevel algorithm for adaptive cubic regularization method recently. The above algorithms all have the same aim: to speed up the computations by making use of the geometric similarity between different models in the hierarchy. 

Numerical results from the papers cited above show that multilevel algorithms can take  advantage of the geometric similarity between different discretizations. In particular, they outperform other state-of-the-art optimization methods, especially for large scale models. However, to the best of our knowledge, no theoretical result exists that rigorously explain these empirical observations. The contributions of this paper are:
\begin{itemize}

\item We provide a complete view of line search multilevel algorithm, and in particular, we connect the general framework of the multilevel algorithm with classical optimization algorithms, such as variable metric methods and block-coordinate type methods.

\item We analyze the Newton-type multilevel model. The key feature of the Newton-type multilevel model is that a coarse model is created from the first and second order information of the fine model. We will call this algorithm the \textbf{Ne}wton-type \textbf{M}ultilevel \textbf{O}ptimization (\textbf{NeMO}). A global convergence analysis of NeMO is provided.

\item We propose to use the composite rate for analysis of the local convergence of NeMO. As we will show later, neither linear convergence nor quadratic convergence is suitable when studying the local convergence of NeMO. 

\item We study the composite rate of NeMO in a case study of infinite dimensional optimization problems. We show that the linear component of the composite rate is inversely proportional to the smoothness of the residual, which agrees with the findings in conventional multigrid methods.
\end{itemize}

The rest of this paper is structured as follows: In Section \ref{sec:NeMOMultiModel} we provide background material for multilevel algorithms. In Section \ref{sec:NeMOconvAnalysis}, we study the convergence of NeMO. We first derive the global convergence rate of NeMO, and then show that NeMO exhibits composite convergence when the current incumbent is sufficiently close to the optimum. A composite convergence rate is defined as a linear combination of linear convergence and quadratic convergence, and we denote $r_1$ and $r_2$ as the coefficient of linear rate and quadratic rate, respectively. In Section \ref{sec:PDE}, we compute $r_1$ in problems arising from discretizations of one-dimensional PDE problems and show the relationship between $r_1$ and the structure of the problem. In Section \ref{sec:NeMONumExp} we illustrate the convergence of NeMO using several numerical examples.

\section{Multilevel Models}\label{sec:NeMOMultiModel}
In this section a broad view of the general multilevel framework will be provided. We start with a basic setting and the core idea of multilevel algorithms in \cite{Gratton2008,Lewis2005,Wen2009}. Then we provide the formulation and details of the core topic of this paper, namely Newton-type multilevel model.

\subsection{Problem Formulation}

In this paper we are interested in solving,
\begin{equation}\label{eq: fine model}
\min_{\mathbf{x}_h \in \mathbb{R}^N} f_h(\mathbf{x}_h), 
\end{equation}
where $\mathbf{x}_h \in \mathbb{R}^N$, and the function $f_h:\mathbb{R}^N \rightarrow \mathbb{R}$ is continuous, differentiable, and strongly convex. We clarify the use of the subscript $h$. Throughout this paper, the lower case $h$ represents an object or property that this is associated with the \textit{\textbf{fine}} model, i.e. the model we actually want to solve. To use multilevel methods, one needs to formulate a hierarchy of models with reduced dimensions called the \textit{\textbf{coarse}} models. We only consider two models in the hierarchy: fine and coarse. In the same manner of using subscript $h$, we assign the upper case $H$ to represent the association with coarse model. We assign $N$ and $n$ ($n \leq N$) to be the dimensions of fine model and coarse model, respectively. For instance, any vector that is within the space $\mathbb{R}^{N}$ is denoted with subscript $h$, and similarly, any vector with subscript $H$ is within the space $\mathbb{R}^{n}$.
\begin{assumption}\label{assm:NeMOLipAssm}
There exists constants $\mu_h$, $L_h$, and $M_h$ such that
\begin{equation}\label{eq:LhMuhDef}
\mu_h \mathbf{I} \preccurlyeq \nabla^2 f_h(\mathbf{x}) \preccurlyeq L_h \mathbf{I}, \quad \forall \mathbf{x}_h\in\mathbb{R}^n,
\end{equation}
and 
\begin{equation}\label{eq:MhDef}
\Vert \nabla^2 f_h (\mathbf{x}_h) - \nabla^2 f_h (\mathbf{y}_h) \Vert \leq M_h \Vert \mathbf{x}_h - \mathbf{y}_h \Vert , \quad \forall \mathbf{x}_h,\mathbf{y}_h\in\mathbb{R}^n .   
\end{equation}
Equation (\ref{eq:LhMuhDef}) implies
\begin{equation*}
\Vert \nabla f_h(\mathbf{x}_h) - \nabla f_h(\mathbf{y}_h) \Vert \leq L_h \Vert \mathbf{x}_h-\mathbf{y}_h \Vert, \quad \forall \mathbf{x}_h,\mathbf{y}_h\in\mathbb{R}^n .
\end{equation*} 
\end{assumption}
The above assumptions will be used throughout the paper.

Multilevel methods require mapping information across different dimensions. To this end, we define a matrix $\mathbf{P} \in \mathbb{R}^{N\times n}$ to be the prolongation operator which maps information from coarse to fine, and we define a matrix $\mathbf{R} \in \mathbb{R}^{n\times N}$ to be the restriction operator which maps information from fine to coarse. We make the following assumption on $\mathbf{P}$ and $\mathbf{R}$.
\begin{assumption}
\label{assp: P R}
The restriction operator $\mathbf{R}$ is the transpose of the prolongation operator $\mathbf{P}$ up to a constant $c$. That is, 
\[
\mathbf{P} = c \mathbf{R}^T, \qquad c>0.
\]
\end{assumption}
Without loss of generality, we take $c=1$ throughout this paper to simplify the use of notation for the analysis. We also assume any useful (non-zero) information in the coarse model will not become zero after prolongation and thus make the following assumption.
\begin{assumption}\label{assp: P}
The prolongation operator $\mathbf{P}$ has full column rank, and so  
\[
\text{rank}(\mathbf{P})=n.
\]
\end{assumption}
Notice that Assumption \ref{assp: P R} and \ref{assp: P} are standard assumptions for multilevel methods \cite{Briggs2000,Hackbusch2003,Wen2009}. Since $\mathbf{P}$ has full column rank, we define the pseudoinverse and its norm
\begin{equation}\label{eq:xiDef}
\mathbf{P}^{+} = (\mathbf{R}\mathbf{P})^{-1}\mathbf{R}, \quad \text{and} \quad \xi = \Vert\mathbf{P}^{+}\Vert.
\end{equation}
The coarse model is constructed in the following manner. Suppose in the $k^{\text{th}}$ iteration we have an incumbent
solution $\mathbf{x}_{h,k}$ and gradient $\nabla f_{h,k} \triangleq \nabla f_h(\mathbf{x}_{h,k})  $, then the corresponding coarse model is,
\begin{equation} \label{eq: coarse model}
\min_{\mathbf{x}_H \in \mathbf{R}^n} \phi_H(\mathbf{x}_H) \triangleq f_H(\mathbf{x}_H)+\langle \mathbf{v}_H,\mathbf{x}_H - \mathbf{x}_{H,0}\rangle,
\end{equation}
where,
\begin{eqnarray*}
\mathbf{v}_H &\triangleq & -\nabla f_{H,0}+ \mathbf{R} \nabla f_{h,k} ,
\end{eqnarray*}
$\mathbf{x}_{H,0} = \mathbf{R} \mathbf{x}_{h,k}$, and $f_H: \mathbb{R}^n \rightarrow \mathbb{R}$ {is a function to be specified later}. Similar to $\nabla f_{h,k}$, we denote $\nabla^2 f_{H,0} \triangleq \nabla^2 f_H(\mathbf{x}_{H,0})$ and $\nabla \phi_{H,0} \triangleq \nabla \phi_H (\mathbf{x}_{H,0})$ to simplify notation. We emphasize the construction of the coarse model (\ref{eq: coarse model}) is well known and it is not original in this paper. See for example \cite{Gratton2008,Lewis2005,Wen2009}. Note that when constructing the coarse model (\ref{eq: coarse model}), one needs to add an additional linear term {to} $f_H(\mathbf{x}_H)$. This linear term ensures the following is satisfied,
\begin{equation}\label{eq:NeMOhHmatch}
\nabla \phi_{H,0} = \mathbf{R} \nabla f_{h,k}.
\end{equation}
For infinite-dimensional optimization problems, one can define $f_h$ and $f_H$ using discretization with different mesh sizes. In general, $f_h$ is  a function that approximates the original problem sufficiently  well, and that can be achieved using a small mesh size. Based on geometric similarity between discretizations with different meshes, $f_h \approx f_H$ even though $n \leq N$. 

However, we want to emphasize $f_h \approx f_H$ is not a necessary requirement when using multilevel methods. In principle, $f_H(\mathbf{x}_H)$ can be any function. Newton-type multilevel model, as we will show later, is a quadratic model where $f_H$ is chosen to be a quadratic approximation of $f_h$ at some $\mathbf{x}_h$.

\subsection{The General Multilevel Algorithm}\label{sec:NeMOBasicAlg}
The main idea of multilevel algorithms is to use the coarse model to compute search directions. When a direction from the coarse model is used we call the iteration a \textit{\textbf{coarse correction step}}. When using coarse correction step, we compute the direction by solving the corresponding coarse model (\ref{eq: coarse model}) and perform the update,
\[
\mathbf{x}_{h,k+1} = \mathbf{x}_{h,k} + \alpha_{h,k} \hat{\mathbf{d}}_{h,k},
\]
with
\begin{equation}\label{eq: RecurStepDef}
\hat{\mathbf{d}}_{h,k} \triangleq \mathbf{P} ( \mathbf{x}_{H,\star} - \mathbf{x}_{H,0} ),
\end{equation}
where $\mathbf{x}_{H,\star}$ is the solution of the coarse model, and $\alpha_{h,k} \in \mathbb{R}^+$ is the stepsize. We clarify that the ``hat'' in $\hat{\mathbf{d}}_{h,k}$ is used to identify a coarse correction step.

We should emphasize that $\mathbf{x}_{H,\star}$ in (\ref{eq: RecurStepDef}) can be replaced by $\mathbf{x}_{H,r}$ for $r=1,2,\dots$, i.e., the incumbent solution of the coarse mode (\ref{eq: coarse model}) after the $r^{\text{th}}$ iterations {of some iterative method}. However, for the purpose of this paper and simplicity, we ignore this case and we let (\ref{eq: RecurStepDef}) be the {(exact)} coarse correction step.

It is known that the coarse correction step $\hat{\mathbf{d}}_{h,k}$ is a descent direction {for $f_h$} if $f_H$ is convex. The following lemma states this argument rigorously. Even though the proof is provided in \cite{Wen2009}, we provide it with our notation for the completeness of this paper.
\begin{lemma}[\cite{Wen2009}]\label{lm:dHatDescent}
If $f_H$ is a convex function, then the coarse correction step is a descent direction {for $f_h$ at $\mathbf{x}_{h,k}$}. In particular, in the $k^{\text{th}}$ iteration,
\[
\nabla f_{h,k}^T \hat{\mathbf{d}}_{h,k} \leq  \phi_{H,\star} - \phi_{H,0} \leq 0.
\]
\end{lemma}
\begin{proof}
\begin{eqnarray*}
\nabla f_{h,k}^T \hat{\mathbf{d}}_{h,k}  &=& \nabla f_{h,k}^T \mathbf{R}^T \left( \mathbf{x}_{H,\star} - \mathbf{x}_{H,0} \right) ,
\\
&=& \left(\mathbf{R}\nabla f_{h,k}\right)^T  \left( \mathbf{x}_{H,\star} - \mathbf{x}_{H,0} \right),
\\
&=& \nabla \phi_{H,0}^T  \left( \mathbf{x}_{H,\star} - \mathbf{x}_{H,0} \right),
\\
&\leq &  \phi_{H,\star} - \phi_{H,0}.
\end{eqnarray*}
as required, where the last inequality holds because $\phi_H$ is a convex function. 
\end{proof}
Even though Lemma \ref{lm:dHatDescent} states that $\hat{\mathbf{d}}_{h,k}$ is a descent direction, using coarse correction step solely is not sufficient to solve the fine model (\ref{eq: fine model}).
\begin{proposition}\label{prop:RecurStpNotEff}
{Assume that $f_H$ is convex.} Suppose $\nabla f_{h,k} \neq 0$ and $\nabla f_{h,k} \in \text{null}(\mathbf{R})$, then the coarse correction step
\[
\hat{\mathbf{d}}_{h,k} = 0.
\]
\end{proposition}
\begin{proof}
From (\ref{eq:NeMOhHmatch}), $\mathbf{x}_{H,\star} = \mathbf{x}_{H,0}$ when $\mathbf{R}\nabla f_{h,k}=0$. Thus, $\hat{\mathbf{d}}_{h,k} = \mathbf{P} (\mathbf{x}_{H,\star} - \mathbf{x}_{H,0}) =0$.
\end{proof}
Recall that $\mathbf{R} \in \mathbb{R}^{n\times N}$, and so for $n <N$, a coarse correction step could be zero and make no progress even when the first order necessary condition $\nabla f_h = 0$ has not been satisfied.

\subsubsection{Fine Correction Step}

Two approaches can be used when coarse correction step is not progressing nor effective. The first approach is to compute directions using standard optimization methods. We call such step the \textit{\textbf{fine correction step}}. As opposed to coarse correction step $\hat{\mathbf{d}}_{h,k}$, we abandon the use of ``hat'' for all fine correction steps and denote them as $\mathbf{d}_{h,k}$'s. To be precise, we can compute $\mathbf{d}_{h,k}$ using the following, 
\begin{eqnarray}
 \mathbf{d}_{h,k}  &=& \arg \min_{\mathbf{d}}   \frac{1}{2}\langle \mathbf{d} , \mathbf{Q} \mathbf{d}  \rangle + \langle \nabla f_{h,k} , \mathbf{d}  \rangle , \nonumber \\ \label{eq:VetMatStpDef}
 &=& - \mathbf{Q}^{-1}  \nabla f_{h,k}.
 \end{eqnarray}
 where $\mathbf{Q} \in \mathbb{R}^{N \times N}$ is a positive definite matrix. When $\mathbf{Q} = \mathbf{I}$, $\mathbf{d}_{h,k}$ is the steepest descent search direction. When $\mathbf{Q} = \nabla^2 f_{h,k}$, $\mathbf{d}_{h,k}$ is the search direction by Newton's method. When $\mathbf{Q}$ is an approximation of the Hessian, then $\mathbf{d}_{h,k}$ is the quasi-Newton search direction.

We perform a fine correction step when a coarse correction step may not be effective. That is, {when one of the following conditions holds:}
\begin{equation}\label{eq:DirectReq}
\Vert\mathbf{R} \nabla f_{h,k} \Vert < \kappa \Vert\nabla f_{h,k}\Vert \quad \text{or} \quad \Vert\mathbf{R} \nabla f_{h,k}\Vert < \epsilon,
\end{equation}
where $\kappa \in \left(0,\min(1,\Vert\mathbf{R}\Vert)\right)$, and $\epsilon \in (0,1)$. The above criteria prevent the use of the coarse model when $\mathbf{x}_{H,0} \approx \mathbf{x}_{H,\star}$, i.e. the coarse correction step $\hat{\mathbf{d}}_{h,k}$ is close to $\mathbf{0}$. We point out that these criteria were also proposed in \cite{Wen2009}. We also make the following assumption on the fine correction step throughout this paper.
\begin{assumption}\label{assum: DirectStp}
There exists strictly positive constants $\nu_h, \zeta_h>0$ such that
\begin{equation*}
\Vert\mathbf{d}_{h,k}\Vert \leq \nu_h \Vert\nabla f_{h,k}\Vert, \quad \text{and} \quad -\nabla f_{h,k}^T \mathbf{d}_{h,k} \geq  \zeta_h \Vert\nabla f_{h,k} \Vert^2,
\end{equation*}
where $\mathbf{d}_{h,k}$ is a fine correction step. As a consequence, there exists a constant $\Lambda_h > 0$ such that
\[
f_{h,k} - f_{h,k+1} \geq  \Lambda_h  \Vert \nabla f_{h,k} \Vert^2,
\]
where $f_{h,k+1}$ is updated using a fine correction step.
\end{assumption}
As we will show later, Assumption \ref{assum: DirectStp} is not restrictive, and $\Lambda_h$ is known for well-known cases like gradient descent, Newton method, etc. Using the combination of fine and coarse correction steps is the standard approach in multilevel methods, especially for PDE-based optimization problems \cite{Gratton2008,Lewis2005,Wen2009}.

\subsubsection{Multiple $\mathbf{P}$'s and $\mathbf{R}$'s}

The second approach to overcome issue of ineffective coarse correction step is by creating multiple coarse models with different $\mathbf{P}$'s and $\mathbf{R}$'s. 
\begin{proposition}\label{prop:MultiPRnessCon}
Suppose $\mathbf{R}_1, \mathbf{R}_2, \dots, \mathbf{R}_p $ are all restriction operators that satisfy Assumption \ref{assp: P R} and \ref{assp: P}, where $\mathbf{R}_i \in \mathbb{R}^{n_i \times N}$ for $i=1,2,\dots ,p$. Denote $\mathcal{S}$ to be a set that contains the rows of $\mathbf{R}_i$'s in $\mathbb{R}^N$, for $i=1,2,\dots ,p$. If
\[
\text{span}(\mathcal{S}) = \mathbb{R}^N,
\]
then for $\nabla f_{h,k} \neq 0$ there exists at least one $\mathbf{R}_j \in \{ \mathbf{R}_i\}_{i=1}^p$ such that
\[
\hat{\mathbf{d}}_{h,k} \neq 0 \quad \text{and} \quad \nabla f_{h,k}^T \hat{\mathbf{d}}_{h,k} < 0,
\]
where $\hat{\mathbf{d}}_{h,k}$ is the coarse correction step computed using $\mathbf{R}_j$.
\end{proposition}
\begin{proof}
Since $\text{span}(\mathcal{S}) = \mathbb{R}^N$, then for $\nabla f_{h,k} \neq 0$, there exists one $\mathbf{R}_j$ such that $\mathbf{R}_j\nabla f_{h,k} \neq 0$. So the corresponding coarse model would have $\mathbf{x}_{H,\star} \neq \mathbf{x}_{H,0}$, and thus $\hat{\mathbf{d}}_{h,k_j} \neq 0$.
\end{proof}
Proposition \ref{prop:MultiPRnessCon} shows that if the rows of the restriction operators $\mathbf{R}_i$'s span $\mathbb{R}^N$, then at least one coarse correction step from these restriction operators would be nonzero and thus effective. In each iteration, one could use the similar idea as in (\ref{eq:DirectReq}) to rule out ineffective coarse models. However, this checking process could be expensive for large scale problems with large $p$ (number of restriction operators). To omit this checking process, one could choose the following alternatives.
\begin{itemize}
	\item[i.] \textbf{Cyclical approach}: choose $\mathbf{R}_1, \mathbf{R}_2, \dots, \mathbf{R}_p$ in order at each iteration, and choose $\mathbf{R}_1$ after $\mathbf{R}_p$.
	\item[ii.] \textbf{Probabilistic approach}: assign a probability mass function with $\{\mathbf{R}_i\}_{i=1}^p$ as a sample space, and choose the coarse model randomly based on the mass function. The mass function has to be strictly positive for each $\mathbf{R}_i$'s.
\end{itemize}
We point out that this idea of using multiple coarse models is related to domain decomposition methods, which solve (non-)linear equations arising from PDEs. Domain decomposition methods partition the problem domain into several sub-domains, and thus decompose the original problem into several smaller problems. We refer the reader to \cite{Chan1994} for more details about domain decomposition methods.

\subsection{Connection with Variable Metric Methods}\label{sec:VarMetMeth}
Using the above multilevel framework, in the rest of this section we will introduce different versions of multilevel algorithms: variable metric methods, block-coordinate descent, and stochastic variance reduced gradient. At the end of this section we will introduce the Newton-type multilevel model, which is an interesting case of the multilevel framework.

Recall that for variable metric methods, the direction $\mathbf{d}_{h,k}$ is computed by solving
\begin{eqnarray}
\mathbf{d}_{h,k}  &=& \arg \min_{\mathbf{d}}   \frac{1}{2}\langle \mathbf{d} , \mathbf{Q} \mathbf{d}  \rangle + \langle \nabla f_{h,k} , \mathbf{d}  \rangle ,
\nonumber \\ \label{eq:VetMatStpDef}
&=& - \mathbf{Q}^{-1}  \nabla f_{h,k}.
\end{eqnarray}
where $\mathbf{Q} \in \mathbb{R}^{N \times N}$ is a positive definite matrix. When $\mathbf{Q} = \mathbf{I}$, $\mathbf{d}_{h,k}$ is the steepest descent search direction. When $\mathbf{Q} = \nabla^2 f_{h,k}$, $\mathbf{d}_{h,k}$ is the search direction by Newton's method. When $\mathbf{Q}$ is an approximation of the Hessian, then $\mathbf{d}_{h,k}$ is the quasi-Newton search direction.

To show the connections between multilevel methods and variable metric methods, consider the following $f_H$.
\begin{equation}\label{eq:ReduceVarMetMet}
f_H(\mathbf{x}_H) = \frac{1}{2} \langle \mathbf{x}_H - \mathbf{x}_{H,0} , \mathbf{Q}_H  (\mathbf{x}_H - \mathbf{x}_{H,0}) \rangle ,
\end{equation}
where $\mathbf{Q}_H \in \mathbb{R}^{n \times n}$, and $\mathbf{x}_{H,0} = \mathbf{R} \mathbf{x}_{h,k}$ as defined in (\ref{eq: coarse model}). Applying the definition of the coarse model (\ref{eq: coarse model}), we obtain,
\begin{equation}\label{eq:ReduceVarMetMetCoarseMod}
\min_{\mathbf{x}_H \in \mathbf{R}^n} \phi_H(\mathbf{x}_H) = \frac{1}{2} \langle \mathbf{x}_H - \mathbf{x}_{H,0} , \mathbf{Q}_H  (\mathbf{x}_H - \mathbf{x}_{H,0}) \rangle +\langle \mathbf{R} \nabla f_{h,k} ,\mathbf{x}_H - \mathbf{x}_{H,0}\rangle.
\end{equation}
Thus from the definition in (\ref{eq: RecurStepDef}), the associated coarse correction step is,
\begin{equation}\label{eq:RvarMetdir}
\hat{\mathbf{d}}_{h,k}  = \mathbf{P} \left(  \arg \min_{\mathbf{d}_H \in \mathbf{R}^n} \underbrace{ \frac{1}{2} \langle \mathbf{d}_H , \mathbf{Q}_H  \mathbf{d}_H \rangle +\langle \mathbf{R} \nabla f_{h,k} ,\mathbf{d}_H \rangle }_{\mathbf{d}_H = \mathbf{x}_H-\mathbf{x}_{H,0}}\right) = - \mathbf{P}\mathbf{Q}_H^{-1}\mathbf{R} \nabla f_{h,k}.
\end{equation}
Therefore, with this specific $f_H$ in (\ref{eq:ReduceVarMetMet}), the resulting coarse model (\ref{eq:ReduceVarMetMetCoarseMod}) is analogous to variable metric methods. In a naive case where $n=N$ and $\mathbf{P} = \mathbf{R} = \mathbf{I}$, the corresponding coarse correction step (\ref{eq:RvarMetdir}) would be the same as steepest descent direction, Newton direction, and quasi-Newton direction for $\mathbf{Q}_H$ that is identity matrix, Hessian, and approximation of Hessian, respectively.
\subsection{Connection with Block-coordinate Descent}\label{sec:NeMABCDintro}
Interestingly, the coarse model (\ref{eq:ReduceVarMetMetCoarseMod}) is also related to block-coordinate type methods. Suppose we have $p$ coarse models with prolongation and restriction operators, $\{\mathbf{P}_i\}_{i=1}^p$ and $\{\mathbf{R}_i\}_{i=1}^p$, respectively. For each coarse model, we let (\ref{eq:ReduceVarMetMet}) be the corresponding $f_H$ with $\mathbf{Q}_H = \mathbf{I}$, and we further restrict our setting with the following properties. 
\begin{itemize}
	\item[1.] $\mathbf{P}_i \in \mathbb{R}^{N \times n_i}$, $\forall i=1,2,\dots,p$.
	\item[2.] $\mathbf{P}_i = \mathbf{R}_i^T$, $\forall i=1,2,\dots,p$.
	\item[3.] $[\mathbf{P}_1 \ \mathbf{P}_2 \dots \mathbf{P}_p] = \mathbf{I}$.
\end{itemize}
From (\ref{eq:RvarMetdir}), the above setting results in $\hat{\mathbf{d}}_{h,k_i} =  - \mathbf{P}_i  \mathbf{R}_i \nabla f_{h,k}$, where $\hat{\mathbf{d}}_{h,k_i}$ is the coarse correction step for the $i^{\text{th}}$ model. Notice that
\[  
(\mathbf{P}_i  \mathbf{R}_i \nabla f_{h,k})_j =
\begin{cases}
(\nabla f_{h,k})_j & \text{if} \ \  \displaystyle\sum_{q=1}^{i-1} n_q < j \leq \displaystyle\sum_{q=1}^{i} n_q ,\\
0                  & \text{otherwise } .
\end{cases}
\]
Therefore, $\hat{\mathbf{d}}_{h,k_i}$ is equivalent to a block-coordinate descent update \cite{Beck2013}. When $n_i=1$, for $i=1,2,\dots,p$, it becomes a coordinate descent method. When $1< n_i<N$, for $i=1,2,\dots,p$, it becomes a block-coordinate descent. When $\mathbf{P}_i$'s and $\mathbf{R}_i$'s are chosen using the cyclical approach, then it would be a cyclical (block)-coordinate descent. When $\mathbf{P}_i$'s and $\mathbf{R}_i$'s are chosen using the probabilistic approach, then it would be a randomized (block)-coordinate descent method.

\subsection{The Newton-type Multilevel Model}

We end this section with the core topic of this paper - the Newton-type multilevel model. The Newton-type multilevel coarse model is a special case of (\ref{eq:ReduceVarMetMetCoarseMod}) where,
\begin{equation}\label{eq:NeMOG-fH}
\mathbf{Q}_H = \nabla_H^2 f_{h,k} \triangleq \mathbf{R} \nabla^2 f_{h,k} \mathbf{P},
\end{equation}
and so the Newton-type multilevel (coarse) model is,
\begin{equation}\label{eq:G-model}
\min_{\mathbf{x}_H \in \mathbf{R}^n} \phi_H(\mathbf{x}_H) = \frac{1}{2} \langle \mathbf{x}_H - \mathbf{x}_{H,0} , \nabla_H^2 f_{h,k} (\mathbf{x}_H - \mathbf{x}_{H,0}) \rangle +\langle \mathbf{R} \nabla f_{h,k} ,\mathbf{x}_H - \mathbf{x}_{H,0}\rangle.
\end{equation}
According to (\ref{eq:RvarMetdir}), the corresponding coarse correction step is
\begin{equation}\label{eq:NeMOStpDef}
\hat{\mathbf{d}}_{h,k} = - \mathbf{P} [\mathbf{R} \nabla^2 f_{h,k} \mathbf{P}]^{-1} \mathbf{R} \nabla f_{h,k} = - \mathbf{P} [\nabla_H^2 f_{h,k}]^{-1} \mathbf{R} \nabla f_{h,k}.
\end{equation}
In the context of multilevel optimization, to the best of our knowledge, this coarse model was first considered in \cite{Gratton2008}. In \cite{Gratton2008} a trust-region type multilevel method is proposed to solve PDE-based optimization problems, and the Newton-type multilevel model is described as a ``radical strategy''. In a later paper from Gratton et al. \cite{Gratton2010}, a trust-region type multilevel method was tested numerically, and the Newton-type multilevel model showed promising numerical results.

It is worth mentioning that the above coarse correction step is equivalent to the solution of the system of linear equations,
\begin{equation}\label{eq:GalnewtonSys}
\mathbf{R}\nabla^2 f_{h,k} \mathbf{P} \mathbf{d}_{H} = -\mathbf{R} \nabla f_{h,k}.
\end{equation}
which is the general case of the Newton's method in which $\mathbf{P} = \mathbf{R} = \mathbf{I}$. Using Assumption \ref{assp: P}, we can show that $\nabla_H^2 f_{h,k}$ is positive definite, and so equation (\ref{eq:GalnewtonSys}) has a unique solution.
\begin{proposition}\label{prop:coarseHbasic}
$\mathbf{R} \nabla^2 f_h(\mathbf{x}_h) \mathbf{P}$ is positive definite, and in particular,
\[
\mu_h\xi^{-2} \mathbf{I} \preceq  \mathbf{R} \nabla^2 f_h(\mathbf{x}_h) \mathbf{P}  \preceq L_h \omega^2 \mathbf{I}
\]
where $\omega = \max \{ \Vert \mathbf{P} \Vert , \Vert \mathbf{R} \Vert \}$ and $\xi = \Vert \mathbf{P}^+ \Vert$.
\end{proposition}
\begin{proof}
\begin{eqnarray*}
\mathbf{x}^T \left(\mathbf{R}\nabla^2 f_h(\mathbf{x}_h)\mathbf{P} \right)\mathbf{x} = (\mathbf{Px})^T \nabla^2 f_h(\mathbf{x}_h) (\mathbf{Px}) \leq L_h \Vert \mathbf{Px}\Vert^2 \leq L_h \omega^2 \Vert \mathbf{x} \Vert^2.
\end{eqnarray*}
Also,
\begin{eqnarray*}
\mathbf{x}^T \left(\mathbf{R}\nabla^2 f_h(\mathbf{x}_h)\mathbf{P} \right)\mathbf{x} = (\mathbf{Px})^T \nabla^2 f_h(\mathbf{x}_h) (\mathbf{Px}) \geq \mu_h \Vert \mathbf{Px} \Vert^2 \geq \frac{\mu_h}{\Vert \mathbf{P}^{+}\Vert^2} \Vert\mathbf{x}\Vert^2 =  \frac{\mu_h}{\xi^2} \Vert\mathbf{x}\Vert^2.
\end{eqnarray*}
So we obtain the desired result.
\end{proof}

\section{Convergence of NeMO}\label{sec:NeMOconvAnalysis}

\begin{algorithm}[t]
\caption{NeMO}
\begin{algorithmic}\label{alg: MLGM}
\STATE \textbf{Input: }$\mathbf{P} \in \mathbb{R}^{N \times n}$ and $\mathbf{R} \in \mathbb{R}^{N\times n}$ which satisfy Assumption \ref{assp: P R} and \ref{assp: P}, $\kappa \in \left(0,\min(1,\Vert\mathbf{R}\Vert)\right) $, $\epsilon$, $\rho_1 \in (0,0.5)$, $\beta_{ls} \in (0,1)$.
\STATE \textbf{Initialization:} $\mathbf{x}_{h,0} \in \mathbb{R}^N$
\FOR{$k=0,1,2,\dots$}
\STATE Compute the direction
\[
\mathbf{d} =
\begin{cases}
 \hat{\mathbf{d}}_{h,k}  \text{  in (\ref{eq:NeMOStpDef})} & \text{if} \ \ \Vert \mathbf{R} \nabla f_{h,k}\Vert > \kappa \Vert \nabla f_{h,k} \Vert \text{ and } \Vert \mathbf{R} \nabla f_{h,k} \Vert > \epsilon,
\\
\mathbf{d}_{h,k}  \text{ in  (\ref{eq:VetMatStpDef})}                 & \text{otherwise} .
\end{cases}
\]
\STATE Find the smallest $q \in \mathbb{N}$ such that for stepsize $\alpha_{h,k} = \beta_{ls}^q$, \[ f_h ( \mathbf{x}_{h,k} + \alpha_{h,k} \mathbf{d} ) \leq f_{h,k} + \rho_1 \alpha_{h,k} \nabla^T f_{h,k} \mathbf{d} .\]
\STATE Update \[\mathbf{x}_{h,k+1} \triangleq \mathbf{x}_{h,k} + \alpha_{h,k} \mathbf{d}.\]
\ENDFOR
\end{algorithmic}
\end{algorithm}

In this section we analyze NeMO (Algorithm \ref{alg: MLGM}). The fine correction steps in Algorithm \ref{alg: MLGM} are deployed by a variable metric method, and an Armijo rule is used as stepsize strategy for both fine and coarse correction steps. We will first show that Algorithm \ref{alg: MLGM} achieves a sublinear rate of convergence. We then analyze the maximum number of coarse correction steps that would be taken by Algorithm \ref{alg: MLGM}, and the condition that when the coarse correction steps yield quadratic reduction in the gradients in the subspace. At the end of this section, we will provide the composite convergence rate for the coarse correction steps.

To provide convergence properties when the coarse correction step is used, the following quantity will be used 
\begin{equation*}
\chi_{H,k} \triangleq [(\mathbf{R}\nabla f_{h,k})^T[\nabla_H^2 f_{h,k}]^{-1}\mathbf{R}\nabla f_{h,k}]^{1/2}.
\end{equation*}
Notice that $\chi_{H,k}$ is analogous to the Newton decrement, which is used to study the convergence of the Newton method \cite{Boyd:2004:CO:993483}. In particular, $\chi_{H,k}$ has the following properties.
\begin{itemize}
\item[1.] $\nabla f_{h,k}^T \hat{\mathbf{d}}_{h,k} =  -\chi_{H,k}^2$.
\item[2.] $\hat{\mathbf{d}}_{h,k}^T \nabla^2 f_{h,k} \hat{\mathbf{d}}_{h,k} =  \chi_{H,k}^2$.
\end{itemize}
We omit the proofs of the above properies since these can be done by using direct computation and the definition of $\chi_{H,k}$.

\subsection{The Sublinear Rate}
We will show that Algorithm \ref{alg: MLGM} will achieve a sublinear rate of convergence. We will deploy the techniques from \cite{Beck2013} and \cite{Boyd:2004:CO:993483}. Starting with the following lemma, we state reduction in function value using coarse correction steps. We would like to clarify that even though NeMO is considered as a special case in \cite{Wen2009}, we take advantage of this simplification and specification to provide analysis with results that are easier to interpret. In particular, the analysis of stepsizes $\alpha_{h,k}$'s in \cite{Wen2009} relies on the maximum number of iterations taken. This result is unfavorable and unnecessary for the setting we consider.

\begin{lemma} \label{lm: RUstepFimprove}
The coarse correction step $\hat{\mathbf{d}}_{h,k}$ in Algorithm \ref{alg: MLGM} will lead to reduction in function value
\[
f_{h,k} - f_h( \mathbf{x}_{h,k} +\alpha_{h,k} \hat{\mathbf{d}}_{h,k}) \geq  \frac{\rho_1 \kappa^2 \beta_{ls} \mu_h}{ \omega^2 L_h^2}   \Vert \nabla f_{h,k} \Vert^2,
\]
where $\rho_1$, $\kappa$, and $\beta_{ls}$ are user-defined parameters in Algorithm \ref{alg: MLGM}.  $L_h$ and $\mu_h$ are defined in Assumption~\ref{assm:NeMOLipAssm}. $\omega$ is defined in Proposition \ref{prop:coarseHbasic}. 
\end{lemma}
\begin{proof}
By convexity,
\begin{eqnarray*}
f_h(\mathbf{x}_{h,k} +\alpha \hat{\mathbf{d}}_{h,k}) &\leq & f_{h,k} + \alpha \langle \nabla f_{h,k}, \hat{\mathbf{d}}_{h,k} \rangle + \frac{L_h}{2} \alpha^2 \Vert\hat{\mathbf{d}}_{h,k}\Vert^2,
\\
&\leq & f_{h,k} - \alpha \chi_{H,k}^2 + \frac{L_h}{2\mu_h} \alpha^2 \chi_{H,k}^2,
\end{eqnarray*}
since
\[
\mu_h\Vert\hat{\mathbf{d}}_{h,k}\Vert^2 \leq \hat{\mathbf{d}}_{h,k}^T\nabla^2 f_h(x_k) \hat{\mathbf{d}}_{h,k} = \chi_{H,k}^2.
\]
Notice that for $\hat{\alpha} = \mu_h/L_h$, we have
\[
- \hat{\alpha} + \frac{L_h}{2\mu_h} \hat{\alpha}^2 = - \hat{\alpha}  + \frac{L_h}{2\mu_h} \frac{\mu_h}{L_h}\hat{\alpha} = -\frac{1}{2}\hat{\alpha},
\]
and
\begin{eqnarray*}
f_h(\mathbf{x}_{h,k} + \hat{\alpha} \hat{\mathbf{d}}_{h,k}) &\leq & f_{h,k} - \frac{\hat{\alpha}}{2}  \chi_{H,k}^2,
\\
&\leq & f_{h,k} + \frac{\hat{\alpha}}{2}  \nabla f_{h,k}^T \hat{\mathbf{d}}_{h,k},
\\
& < & f_{h,k} + \rho_1 \hat{\alpha}  \nabla f_{h,k}^T \hat{\mathbf{d}}_{h,k},
\end{eqnarray*}
which satisfies the Armijo condition. Therefore, line search will return stepsize $\alpha_{h,k} \geq \hat{\alpha} = (\beta_{ls}\mu_h)/L_h$. Using the fact that
\[
\frac{1}{\omega^2 L_h}\Vert \mathbf{R} \nabla f_{h}(x_k)\Vert^2 \leq (\mathbf{R}\nabla f_{h,k})^T[\nabla_H^2 f_{h,k}]^{-1}\mathbf{R}\nabla f_{h,k}= \chi_{H,k}^2,
\]
we obtain
\begin{eqnarray*}
f_h(\mathbf{x}_{h,k} + \alpha_{h,k} \hat{\mathbf{d}}_{h,k}) - f_{h,k} & \leq &  \rho_1 \alpha_{h,k}  \nabla f_{h,k}^T \hat{\mathbf{d}}_{h,k},
\\
& \leq & - \rho_1 \hat{\alpha}  \chi_{H,k}^2,
\\
& \leq & - \rho_1 \frac{\beta_{ls}\mu_h}{\omega^2 L_h^2}  \Vert \mathbf{R} \nabla f_{h,k} \Vert^2,
\\
& \leq & - \frac{\rho_1 \kappa^2 \beta_{ls} \mu_h}{\omega^2 L_h^2}  \Vert \nabla f_{h,k} \Vert^2,
\end{eqnarray*}
as required.
\end{proof}
Using the result in Lemma \ref{lm: RUstepFimprove}, we derive the guaranteed reduction in function value in the following two lemmas.
\begin{lemma} \label{lm: AllstepFimprove}
Let $\Lambda \triangleq \min\left\{ \Lambda_h, \dfrac{\rho_1 \kappa^2 \beta_{ls} \mu_h}{\omega^2 L_h^2} \right\}$, then the step $\mathbf{d}$ in Algorithm \ref{alg: MLGM} will lead to
\[
f_{h,k} - f_{h,k+1} \geq  \Lambda  \Vert\nabla f_{h,k}\Vert^2,
\]
where $\rho_1$, $\kappa$, and $\beta_{ls}$ are user-defined parameters in Algorithm \ref{alg: MLGM}. $L_h$ and $\mu_h$ are defined in Assumption~\ref{assm:NeMOLipAssm}. $\Lambda_h$ is defined in Assumption \ref{assum: DirectStp}. $\omega$ is defined in Proposition \ref{prop:coarseHbasic}. 
\end{lemma}
\begin{proof}
This is a direct result from Lemma \ref{lm: RUstepFimprove} and Assumption \ref{assum: DirectStp}.
\end{proof}

{Let $\mathbf{x}_{h,\star}$ denote the exact solution of (\ref{eq: fine model}) and let $f_{h,\star}\triangleq f(\mathbf{x}_{h,\star})$.}
\begin{lemma}\label{lm:NeMOFunValDiff} 
Suppose 
\[
\mathcal{R}(\mathbf{x}_{h,0}) \triangleq \max_{\mathbf{x}_h \in \mathbb{R}^N} \{ \Vert \mathbf{x}_h - \mathbf{x}_{h,\star} \Vert : f_h (\mathbf{x}_h) \leq f_h (\mathbf{x}_{h,0}) \},
\]
the step in Algorithm \ref{alg: MLGM} will guarantee
\[
f_{h,k} - f_{h,k+1} \geq \frac{\Lambda}{\mathcal{R}^2(\mathbf{x}_{h,0})} \left(f_{h,k} - f_{h,\star} \right)^2,
\]
where $\Lambda$ is defined in Lemma \ref{lm: AllstepFimprove}.
\end{lemma}
\begin{proof}
By convexity, for $k=0,1,2,\dots$,
\begin{eqnarray*}
f_{h,k} - f_{h,\star} &\leq &   \langle \nabla f_{h,k},\mathbf{x}_{h,k}-\mathbf{x}_{h,\star} \rangle ,
\\
&\leq &   \Vert \nabla f_{h,k}\Vert \ \Vert \mathbf{x}_{h,k} -\mathbf{x}_{h,\star}\Vert ,
\\
&\leq &  \mathcal{R}(\mathbf{x}_{h,0}) \Vert \nabla f_{h,k} \Vert  .
\end{eqnarray*}

Using Lemma \ref{lm: AllstepFimprove}, we have
\begin{eqnarray*}
f_{h,k} - f_{h,\star} &\leq & \mathcal{R}(\mathbf{x}_{h,0}) \sqrt{\Lambda^{-1}\left(f_{h,k} - f_{h,k+1} \right)},
\\
\left(\frac{f_{h,k} - f_{h,\star}}{\mathcal{R}(\mathbf{x}_{h,0})}\right)^2 &\leq & \Lambda^{-1}\left(f_{h,k} - f_{h,k+1} \right),
\\
\Lambda \left(\frac{f_{h,k} - f_{h,\star}}{\mathcal{R}(\mathbf{x}_{h,0})}\right)^2 &\leq & f_{h,k} - f_{h,k+1} ,
\end{eqnarray*}
as required.
\end{proof}

The constant $\Lambda$ in Lemma \ref{lm:NeMOFunValDiff} depends on $\Lambda_h$, which is introduced in Assumption \ref{assum: DirectStp}. This constant depends on both the fine correction step chosen and the user-defined parameter $\rho_1$ in Armijo rule. For instance,
\begin{equation*}
\Lambda_h = 
\begin{cases}
\dfrac{\rho_1\mu_h}{L_h^2} & \text{if} \quad \mathbf{d}_{h,k} = - [\nabla^2 f_{h,k}]^{-1}\nabla f_{h,k},\\
\dfrac{\rho_1}{L_h} & \text{if} \quad \mathbf{d}_{h,k} = - \nabla f_{h,k}.
\end{cases}
\end{equation*}
The above results can be derived via direct computation on bounding the Armijo condition. In order to derive the convergence rate in this section, we use the following lemma on nonnegative scalar sequences.
\begin{lemma}{\cite{Beck2013}}\label{lm:bkseqLem}
Let $\{A_k\}_{k\geq 0}$ be a nonnegative sequence of the real numbers satisfying 
\[
A_k - A_{k+1} \geq \gamma A_k^2 , \quad k=0,1,2,\dots,
\] 
and 
\[
A_0 \leq \frac{1}{q\gamma}
\]
for some positive $\gamma$ and $q$. Then
\[
A_k \leq \frac{1}{\gamma(k+q)}, \quad	k=0,1,2,\dots,
\]
and so
\[
A_k \leq \frac{1}{\gamma k}, \quad	k=0,1,2,\dots .
\]
\end{lemma}
\begin{proof}
See Lemma 3.5 in \cite{Beck2013}.
\end{proof}
Combining the above results, we obtain the rate of convergence.   
\begin{theorem}\label{thm: worseLinCov}
Let $\{\mathbf{x}_k\}_{k\geq 0}$ be the sequence that is generated by Algorithm \ref{alg: MLGM}. Then,
\[
f_{h,k}- f_{h,\star} \leq \frac{\mathcal{R}^2(\mathbf{x}_{h,0})}{\Lambda}\frac{1}{2+k},
\]
where $\Lambda$ and $\mathcal{R}(\cdot)$ are defined as in Lemma \ref{lm: AllstepFimprove} and \ref{lm:NeMOFunValDiff}, respectively.
\end{theorem}
\begin{proof}
From Lemma \ref{lm:NeMOFunValDiff},
\[
f_{h,k} - f_{h,k+1} \geq \frac{\Lambda}{\mathcal{R}^2(\mathbf{x}_{h,0})} \left(f_{h,k} - f_{h,\star} \right)^2.
\]
and so
\[
(f_{h,k} - f_{h,\star}) - (f_{h,k+1} - f_{h,\star}) \geq \frac{\Lambda}{\mathcal{R}^2(\mathbf{x}_{h,0})} \left(f_{h,k} - f_{h,\star} \right)^2.
\]
Also, we have
\begin{eqnarray*}
f_{h,0} - f_{h,\star} \leq \frac{L_h}{2} \Vert \mathbf{x}_{h,0} - \mathbf{x}_{h,\star} \Vert^2 
\leq \frac{L_h}{2}\mathcal{R}^2(\mathbf{x}_{h,0})
\leq \frac{L_h^2\mathcal{R}^2(\mathbf{x}_{h,0})}{2\mu_h}
&\leq &\frac{L_h^2\mathcal{R}^2(\mathbf{x}_{h,0})}{2\mu_h\beta_{ls}\kappa^2 \rho_1},
\\
&\leq &\frac{\mathcal{R}^2(\mathbf{x}_{h,0})}{2\Lambda},
\end{eqnarray*}
where the first inequality holds because of first order condition and the definition of $L_h$ in Assumption \ref{assm:NeMOLipAssm}. Let's $A_k \triangleq f_{h,k}- f_{h,\star}$, $\gamma \triangleq \dfrac{\Lambda}{{R}^2(\mathbf{x}_{h,0})}$, and $q \triangleq 2$. By applying Lemma \ref{lm:bkseqLem}, we have
\[
f_{h,k}- f_{h,\star} \leq \frac{\mathcal{R}^2(\mathbf{x}_{h,0})}{\Lambda}\frac{1}{2+k},
\]
as required.
\end{proof}
Theorem \ref{thm: worseLinCov} provides the sublinear convergence of Algorithm \ref{alg: MLGM}. We emphasize that the rate is inversely proportional to $\Lambda = \min \{\Lambda_h , \rho_1 \kappa^2 \mu_h/ L_h^2\}$, and so small $\kappa$ would result in slow convergence. Therefore, even though $\kappa$ could be arbitrary small, it is not desirable in terms of worse case complexity. Note that $\kappa$ is a user-defined parameter for determining whether the coarse correction step should be used. If $\kappa$ is chosen to be too large, then it is less likely that the coarse correction step would be used. In the extreme case where $\kappa \geq \Vert \mathbf{R} \Vert$, the coarse correction step would not be deployed because,
\[
\Vert \mathbf{R} \nabla f_{h,k} \Vert \leq \Vert \mathbf{R} \Vert  \Vert\nabla f_{h,k} \Vert ,
\]
and so Algorithm \ref{alg: MLGM} reduces to the standard variable metric method. Therefore, there is a trade-off between the worse case complexity and the likelihood that the coarse correction step is deployed.

\subsection{Maximum Number of Iterations of Coarse Correction Step}
We now discuss the maximum number of coarse correction steps in Algorithm \ref{alg: MLGM}. The following lemma will state the sufficient conditions for not taking any coarse correction step.
\begin{lemma}
No coarse correction step in Algorithm \ref{alg: MLGM} will be taken when
\[
\Vert \nabla f_{h,k} \Vert \leq  \frac{\epsilon}{\omega},
\]
where $\omega = \max \{ \Vert \mathbf{P} \Vert , \Vert \mathbf{R} \Vert \}$, and $\epsilon$ is a user-defined parameter in Algorithm \ref{alg: MLGM}.
\end{lemma}
\begin{proof}
Recall that in Algorithm \ref{alg: MLGM}, the coarse step is only taken when $\Vert \mathbf{R} \nabla f_{h,k} \Vert > \epsilon $. We have,
\begin{equation*}
\Vert \mathbf{R} \nabla f_{h,k} \Vert \leq  \omega \Vert \nabla f_{h,k} \Vert \leq  \omega \frac{\epsilon}{\omega} = \epsilon,
\end{equation*}
and so no coarse correction step will be taken.
\end{proof}
The above lemma states the condition when the coarse correction step would not be performed. We then investigate the maximum number of iterations to achieve that sufficient condition.

\begin{lemma}\label{lm:NeMOWorseCount}
Let $\{\mathbf{x}_k\}_{k \geq 0}$ be a sequence generated by Algorithm \ref{alg: MLGM}. Then, $\forall \bar{\epsilon}, \bar{k} >0 $ such that,
\[
\bar{k} \geq \left(\frac{1}{\bar{\epsilon}}\right)^2\frac{\mathcal{R}^2(\mathbf{x}_{h,0})}{\Lambda^2}-2,
\]
we obtain
\[
\Vert \nabla f_{h} (\mathbf{x}_{h,\bar{k}})\Vert \leq \bar{\epsilon},
\]
where $\Lambda$ and $\mathcal{R}(\cdot)$ are defined as in Lemma \ref{lm: AllstepFimprove} and \ref{lm:NeMOFunValDiff}, respectively.
\end{lemma}
\begin{proof}
From Lemma \ref{lm: AllstepFimprove}, we know that
\begin{eqnarray*}
\Lambda  \Vert\nabla f_{h,k} \Vert^2 &\leq & f_{h,k} - f_{h,k+1} .
\end{eqnarray*}
Also, from Theorem \ref{thm: worseLinCov}, we have,
\[
f_{h,k}- f_{h,\star} \leq \frac{\mathcal{R}^2(\mathbf{x}_{h,0})}{\Lambda}\frac{1}{2+k}.
\]
Therefore,
\begin{eqnarray*}
\Vert\nabla f_{h,k} \Vert^2 &\leq & \frac{1}{\Lambda} \left( f_{h,k} - f_{h,k+1}\right) ,
\\
&\leq & \frac{1}{\Lambda} \left( f_{h,k}- f_{h,\star} \right) ,
\\
&\leq &  \frac{\mathcal{R}^2(\mathbf{x}_{h,0})}{\Lambda^2}\frac{1}{2+k}  .
\end{eqnarray*}
For
\[
k = \left(\frac{1}{\bar{\epsilon}}\right)^2\frac{\mathcal{R}^2(\mathbf{x}_{h,0})}{\Lambda^2}-2 ,
\]
we have 
\begin{equation*}
\Vert\nabla f_{h,k} \Vert \leq \sqrt{\frac{\mathcal{R}^2(\mathbf{x}_{h,0})}{\Lambda^2}\frac{1}{2+k}} \leq \sqrt{\frac{\mathcal{R}^2(\mathbf{x}_{h,0})}{\Lambda^2} \left(\bar{\epsilon}\right)^2\frac{\Lambda^2}{\mathcal{R}^2(\mathbf{x}_{h,0})}} = \bar{\epsilon},
\end{equation*}
as required.
\end{proof}
By integrating the above results, we obtain the maximum number of iterations to achieve $\Vert \nabla f_{h,k} \Vert \leq \epsilon/\omega $. That is, no coarse correction step will be taken after
\[
\left(\frac{\omega}{\epsilon}\right)^2\frac{\mathcal{R}^2(\mathbf{x}_{h,0})}{\Lambda^2}-2 \quad \text{iterations}.
\]
Notice that the smaller $\epsilon$, the more coarse correction step will be taken. Depending on the choice of $\mathbf{d}_{h,k}$, the choice of $\epsilon$ could be different. For example, if $\mathbf{d}_{h,k}$ is chosen as the Newton step where $\mathbf{d}_{h,k} = -[\nabla^2 f_{h,k}]^{-1} \nabla f_{h,k}$, one good choice of $\epsilon$ could be $3 \omega (1-2\rho_1) \mu_h^2 /L_h$ if $\mu_h$ and $L_h$ are known. This is because Newton's method achieves quadratic rate of convergence when $\Vert \nabla f_{h,k} \Vert \leq 3  (1-2\rho_1) \mu_h^2 /L_h$ \cite{Boyd:2004:CO:993483}. Therefore, for such $\epsilon$, no coarse correction step would be taken when the Newton method is in its quadratically convergent phase.

\subsection{Quadratic Phase in Subspace}
We now state the required condition for stepsize $\alpha_{h,k}=1$, and then we will show that when $\Vert \mathbf{R}\nabla f_{h,k}\Vert$ is sufficiently small, the coarse correction step would reduce $ \Vert \mathbf{R}\nabla f_{h,k}\Vert$ quadratically. The results below are analogous to the analysis of the Newton's method in \cite{Boyd:2004:CO:993483}. 

\begin{lemma}\label{lm: stepsize1}
Suppose coarse correction step $\hat{\mathbf{d}}_{h,k}$ in Algorithm \ref{alg: MLGM} is taken, then $\alpha_{h,k} = 1$ when
\[
\Vert \mathbf{R} \nabla f_{h,k}\Vert  \leq \eta = \dfrac{3\mu_h^2}{ M_h} (1-2\rho_1),
\] 
where $\rho_1$ is an user-defined parameter in Algorithm \ref{alg: MLGM}. $M_h$ and $\mu_h$ are defined in Assumption \ref{assm:NeMOLipAssm}. 
\end{lemma}
\begin{proof}
By Lipschitz continuity (\ref{eq:MhDef}),
\begin{equation*}
\Vert \nabla^2 f_h( \mathbf{x}_{h,k}+\alpha \hat{\mathbf{d}}_{h,k}) - \nabla^2 f_{h,k} \Vert \leq \alpha M_h \Vert \hat{\mathbf{d}}_{h,k}\Vert,
\end{equation*}
which implies
\begin{equation*}
\Vert \hat{\mathbf{d}}_{h,k}^T( \nabla^2 f_h(\mathbf{x}_{h,k}+\alpha \hat{\mathbf{d}}_{h,k}) - \nabla^2 f_{h,k} )\hat{\mathbf{d}}_{h,k}\Vert  \leq \alpha M_h \Vert \hat{\mathbf{d}}_{h,k}\Vert ^3.
\end{equation*}
Let $\tilde{f}(\alpha) = f_h(\mathbf{x}_{h,k}+\alpha \hat{\mathbf{d}}_{h,k})$, then the above inequality can be rewritten as
\begin{equation*}
| \tilde{f}''(\alpha) - \tilde{f}''(0) | \leq \alpha M_h \Vert \hat{\mathbf{d}}_{h,k}\Vert ^3,
\end{equation*}
and so 
\begin{equation*}
 \tilde{f}''(\alpha) \leq \tilde{f}''(0) + \alpha M_h \Vert \hat{\mathbf{d}}_{h,k}\Vert^3.
\end{equation*}
Since $\tilde{f}''(0) = \hat{\mathbf{d}}_{h,k}^T \nabla^2 f_{h,k} \hat{\mathbf{d}}_{h,k} = \chi_{H,k}^2$,
\begin{equation*}
 \tilde{f}''(\alpha) \leq \chi_{H,k}^2 + \alpha M_h \Vert \hat{\mathbf{d}}_{h,k} \Vert^3.
\end{equation*}
By integration,
\begin{equation*}
 \tilde{f}'(\alpha) \leq \tilde{f}'(0) + \alpha \chi_{H,k}^2 + (\alpha^2/2) M_h \Vert \hat{\mathbf{d}}_{h,k} \Vert^3.
\end{equation*}
Similarly, $\tilde{f}'(0) = \nabla f_{h,k}^T \hat{\mathbf{d}}_{h,k} = -\chi_{H,k}^2$, and so
\begin{equation*}
 \tilde{f}'(\alpha) \leq -\chi_{H,k}^2 + \alpha \chi_{H,k}^2 + (\alpha^2/2) M_h \Vert \hat{\mathbf{d}}_{h,k} \Vert^3.
\end{equation*}
Integrating the above inequality, we obtain
\begin{equation*}
\tilde{f}(\alpha) \leq \tilde{f}(0) - \alpha\chi_{H,k}^2 + (\alpha^2/2) \chi_{H,k}^2 + (\alpha^3/6) M_h \Vert \hat{\mathbf{d}}_{h,k} \Vert^3.
\end{equation*}
Recall that $\mu_h \Vert \hat{\mathbf{d}}_{h,k}\Vert^2 \leq \hat{\mathbf{d}}_{h,k}^T \nabla^2 f_{h,k} \hat{\mathbf{d}}_{h,k} = \chi_{H,k}^2$; thus,
\begin{equation*}
\tilde{f}(\alpha) \leq \tilde{f}(0) - \alpha\chi_{H,k}^2 + \frac{\alpha^2}{2} \chi_{H,k}^2 + \frac{\alpha^3 M_h}{6 \mu_h^{3/2}} \chi_{H,k}^3.
\end{equation*}
Let $\alpha = 1$,
\begin{eqnarray*}
\tilde{f}(1) - \tilde{f}(0) &\leq & - \chi_{H,k}^2 + \frac{1}{2} \chi_{H,k}^2 + \frac{M_h}{6 \mu_h^{3/2}} \chi_{H,k}^3 , 
\\
&\leq & - \left( \frac{1}{2} - \frac{M_h}{6 \mu_h^{3/2}} \chi_{H,k} \right) \chi_{H,k}^2 . 
\end{eqnarray*}
Using the fact that
\[
\Vert \mathbf{R} \nabla f_{h,k}\Vert  \leq \eta = \dfrac{3\mu_h^2}{ M_h} (1-2\rho_1),
\] 
and  
\[
\chi_{H,k} = ((\mathbf{R}\nabla f_{h,k})^T[\nabla_H^2 f_{h,k}]^{-1} \mathbf{R}\nabla f_{h,k})^{1/2} \leq \frac{1}{\sqrt{\mu_h}} \Vert \mathbf{R} \nabla f_{h,k}\Vert ,
\]
we have 
\[
\chi_{H,k} \leq \frac{3\mu_h^{3/2}}{M_h} (1-2\rho_1) \quad \Longleftrightarrow \quad  \rho_1 \leq \frac{1}{2} - \frac{M_h}{6\mu_h^{3/2}} \chi_{H,k}.
\]
Therefore,
\begin{equation*}
\tilde{f}(1) - \tilde{f}(0) \leq   - \rho_1 \chi_{H,k}^2 = \rho_1 \nabla f_{h,k}^T \hat{\mathbf{d}}_{h,k}, 
\end{equation*}
and we have $\alpha_{h,k} = 1$ when $ \Vert \mathbf{R} \nabla f_{h,k}\Vert  \leq \eta$.
\end{proof}
The above lemma yields the following theorem.

\begin{theorem}\label{thm: convex Quad NewP}
Suppose the coarse correction step $\hat{\mathbf{d}}_{h,k}$ in Algorithm \ref{alg: MLGM} is taken and $\alpha_{h,k}=1$, then
\[
\Vert  \mathbf{R} \nabla f_{h,k+1}\Vert  \leq \frac{\omega^3 \xi^4 M_h}{2 \mu_h^2} \Vert \mathbf{R} \nabla f_{h,k}\Vert^2,
\]
where $M_h$ and $\mu_h$ are defined in Assumption \ref{assm:NeMOLipAssm}, $\omega = \max \{ \Vert \mathbf{P} \Vert , \Vert \mathbf{R} \Vert \}$ and $\xi = \Vert \mathbf{P}^+ \Vert$.
\end{theorem}
\begin{proof}
Since $\alpha_{h,k}=1$, we have
\begin{eqnarray*}
\Vert  \mathbf{R} \nabla f_{h,k+1}\Vert  &=& \Vert \mathbf{R} \nabla f_h(\mathbf{x}_{h,k}+ \hat{\mathbf{d}}_{h,k}) - \mathbf{R} \nabla f_{h,k} - \mathbf{R} \nabla^2 f_{h,k} \mathbf{P} \tilde{\mathbf{d}}_{H,i^{\star}} \Vert 
\\
&\leq & \Vert \mathbf{R} \Vert  \ \Vert \nabla f_h(\mathbf{x}_{h,k}+ \hat{\mathbf{d}}_{h,k}) -  \nabla f_{h,k} - \nabla^2 f_{h,k} \hat{\mathbf{d}}_{h,k} \Vert 
\\
&\leq & \omega \Bigg\vert \Bigg\vert \int_0^1  (\nabla^2 f_h(\mathbf{x}_{h,k}+ t \hat{\mathbf{d}}_{h,k}) - \nabla^2 f_{h,k} ) \hat{\mathbf{d}}_{h,k} \ dt \Bigg\vert \Bigg\vert
\\
&\leq & \omega \frac{M_h}{2} \Vert \hat{\mathbf{d}}_{h,k} \Vert^2,
\end{eqnarray*} 
where $\tilde{\mathbf{d}}_{H,i^{\star}}$ is the direction $\hat{\mathbf{d}}_{h,k}$ at coarse level, i.e. $\mathbf{P} \tilde{\mathbf{d}}_{H,i^{\star}} = \hat{\mathbf{d}}_{h,k}$. Notice that 
\begin{eqnarray*}
\Vert \hat{\mathbf{d}}_{h,k} \Vert  &=& \Vert  \mathbf{P} [ \mathbf{R} \nabla^2 f_{h,k} \mathbf{P}]^{-1} \mathbf{R} \nabla f_{h,k}\Vert  
\\
&\leq & \Vert  \mathbf{P} \Vert  \ \Vert [ \mathbf{R} \nabla^2 f_{h,k} \mathbf{P}]^{-1}\Vert  \ \Vert \mathbf{R}\nabla f_{h,k}\Vert 
\\
&\leq & \frac{\omega \xi^2}{\mu_h} \Vert \mathbf{R} \nabla f_{h,k}\Vert .
\end{eqnarray*}
Thus,
\[
\Vert  \mathbf{R} \nabla f_{h,k+1}\Vert  \leq \frac{\omega^3 \xi^4 M_h}{2 \mu_h^2} \Vert \mathbf{R} \nabla f_{h,k}\Vert^2,
\]
as required.
\end{proof}
The above theorem states the quadratic convergence of $\Vert \nabla f_{h,k} \Vert$ within the subspace $\text{range}(\mathbf{R}) $. However, it does not give insight in the convergence behaviour on the full space $\mathbb{R}^N$. To address this, we study the composite rate of convergence in the next section.

\subsection{Composite Convergence Rate}\label{sec:NeMOCompostRateIntro}
At the end of this section, we study the convergence properties of the coarse correction step when the incumbent is sufficiently close to the solution. In particular, we deploy the idea of composite convergence rate in \cite{Erdogdu2015}, and consider the convergence of the coarse correction step as a combination of linear and quadratic convergence.

The reason of proving composite convergence is due to the broadness of NeMO. Suppose that $\mathbf{P}=\mathbf{R}=\mathbf{I}$, then the coarse correction step in NeMO becomes Newton's method. In such case we expect quadratic convergence when the incumbent is sufficiently close to the solution. On the other hand, suppose $\mathbf{P}$ is any column of $\mathbf{I}$ and $\mathbf{R} = \mathbf{P}^T$, then the coarse correction step is a (weighted) coordinate descent direction. One should not expect more than linear convergence in that case. Therefore, both quadratic convergence and linear convergence are not suitable for NeMO, and one needs the combination of them. In this paper, we propose to use a composite convergence, and show that it can better explain the convergence of NeMO.

We would like to emphasize the difference between our setting compared to \cite{Erdogdu2015}. To the best of our knowledge, composite convergence rate was used in \cite{Erdogdu2015} to study subsampled Newton methods for machine learning problems without dimensionality reduction. In this paper, the class of problems that we consider is not restricted to machine learning, and we focus on the Newton-type multilevel model, which is a reduced dimension model. The results presented in this section are not direct results of the approach in \cite{Erdogdu2015}. In particular, if the exact analysis of \cite{Erdogdu2015} is taken, the derived composite rate would not be useful in our setting, because the coefficient of the linear component would be greater than $1$.

\begin{theorem}\label{thm: composConv}
Suppose the coarse correction step $\hat{\mathbf{d}}_{h,k}$ in Algorithm \ref{alg: MLGM} is taken and $\alpha_{h,k}=1$, then
\begin{multline}
\Vert \mathbf{x}_{h,k+1} - \mathbf{x}_{h,\star} \Vert \leq \Vert \mathbf{I} - \mathbf{P} [\nabla_H^2 f_{h,k}]^{-1} \mathbf{R} \nabla^2 f_{h,k} \Vert  \Vert (\mathbf{I} - \mathbf{PR})(\mathbf{x}_{h,k} - \mathbf{x}_{h,\star}) \Vert 
\\
+  \frac{M_h \omega^2 \xi^2}{2\mu_h} \Vert \mathbf{x}_{h,k} - \mathbf{x}_{h,\star} \Vert^2 ,
\end{multline}
where $M_h$ and $\mu_h$ are defined in Assumption \ref{assm:NeMOLipAssm}, $\omega = \max \{ \Vert \mathbf{P} \Vert , \Vert \mathbf{R} \Vert \}$ and $\xi = \Vert \mathbf{P}^+ \Vert$. The operator $\nabla^2_H$ is defined in (\ref{eq:NeMOG-fH}).
\end{theorem}
\begin{proof}
Denote 
\[
\tilde{\mathbf{Q}} = \int_0^1 \nabla^2 f(\mathbf{x}_{h,\star} - t (\mathbf{x}_{h,k} - \mathbf{x}_{h,\star})) \text{d}t ,
\]
we have
\begin{eqnarray*}
\mathbf{x}_{h,k+1} - \mathbf{x}_{h,\star} &=& \mathbf{x}_{h,k} -\mathbf{x}_{h,\star} - \mathbf{P} [\nabla_H^2 f_{h,k}]^{-1} \mathbf{R} \nabla f_{h,k} ,
\\
& = & \mathbf{x}_{h,k} -\mathbf{x}_{h,\star}  - \mathbf{P} [\nabla_H^2 f_{h,k}]^{-1} \mathbf{R} \tilde{\mathbf{Q}} (\mathbf{x}_{h,k} - \mathbf{x}_{h,\star})  ,
\\
 &=&  \left( \mathbf{I} - \mathbf{P} [\nabla_H^2 f_{h,k}]^{-1} \mathbf{R}  \tilde{\mathbf{Q}}  \right) (\mathbf{x}_{h,k} - \mathbf{x}_{h,\star}),
\\
&=& \left( \mathbf{I} - \mathbf{P} [\nabla_H^2 f_{h,k}]^{-1} \mathbf{R} \nabla^2 f_{h,k}  \right) (\mathbf{x}_{h,k} - \mathbf{x}_{h,\star})
\\
& & \: + \left( \mathbf{P} [\nabla_H^2 f_{h,k}]^{-1} \mathbf{R} \nabla^2 f_{h,k} -  \mathbf{P} [\nabla_H^2 f_{h,k}]^{-1} \mathbf{R} \tilde{\mathbf{Q}} \right) (\mathbf{x}_{h,k} - \mathbf{x}_{h,\star}),
\\
&=& \left( \mathbf{I} - \mathbf{P} [\nabla_H^2 f_{h,k}]^{-1} \mathbf{R} \nabla^2 f_{h,k}  \right)(\mathbf{I} - \mathbf{PR})(\mathbf{x}_{h,k} - \mathbf{x}_{h,\star})
\\
& & \: + \mathbf{P} [\nabla_H^2 f_{h,k}]^{-1} \mathbf{R} \left(  \nabla^2 f_{h,k} -   \tilde{\mathbf{Q}} \right) (\mathbf{x}_{h,k} - \mathbf{x}_{h,\star}).
\end{eqnarray*}
Note that
\[
\Vert \nabla^2 f_{h,k} -   \tilde{\mathbf{Q}} \Vert = \Bigg\Vert \nabla^2 f_{h,k} -   \int_0^1 \nabla^2 f(\mathbf{x}_{h,\star} - t (\mathbf{x}_{h,k} - \mathbf{x}_{h,\star})) \text{d}t \Bigg\Vert \leq \frac{M_h}{2} \Vert \mathbf{x}_{h,k} - \mathbf{x}_{h,\star} \Vert .
\]
Therefore,
\begin{eqnarray*}
\Vert \mathbf{x}_{h,k+1} - \mathbf{x}_{h,\star} \Vert &\leq & \Vert \mathbf{I} - \mathbf{P} [\nabla_H^2 f_{h,k}]^{-1} \mathbf{R} \nabla^2 f_{h,k} \Vert  \Vert (\mathbf{I} - \mathbf{PR})(\mathbf{x}_{h,k} - \mathbf{x}_{h,\star}) \Vert 
\\
&& \: + \Vert \mathbf{P} [\nabla_H^2 f_{h,k}]^{-1} \mathbf{R} \Vert \frac{M_h}{2} \Vert \mathbf{x}_{h,k} - \mathbf{x}_{h,\star} \Vert^2 ,
\\
& \leq & \Vert \mathbf{I} - \mathbf{P} [\nabla_H^2 f_{h,k}]^{-1} \mathbf{R} \nabla^2 f_{h,k} \Vert  \Vert (\mathbf{I} - \mathbf{PR})(\mathbf{x}_{h,k} - \mathbf{x}_{h,\star}) \Vert 
\\
&& \: +  \frac{M_h \omega^2 \xi^2}{2\mu_h} \Vert \mathbf{x}_{h,k} - \mathbf{x}_{h,\star} \Vert^2 ,
\end{eqnarray*}
as required.
\end{proof}
Theorem \ref{thm: composConv} provides the composite convergence rate for the coarse correction step. However, some terms remain unclear, in particular $\Vert \mathbf{I} - \mathbf{P} [\nabla_H^2 f_{h,k}]^{-1} \mathbf{R} \nabla^2 f_{h,k} \Vert$. Notice that in the case when $\text{rank}(\mathbf{P}) = N$ (i.e. $\mathbf{P}$ is invertible), 
\begin{eqnarray*}
\Vert \mathbf{I} - \mathbf{P} [\nabla_H^2 f_{h,k}]^{-1} \mathbf{R} \nabla^2 f_{h,k} \Vert &=& \Vert \mathbf{I} - \mathbf{P} [\mathbf{R} \nabla^2 f_{h,k} \mathbf{P}]^{-1} \mathbf{R} \nabla^2 f_{h,k} \Vert ,
\\
& = &\Vert \mathbf{I} - \mathbf{P} \mathbf{P}^{-1} [\nabla^2 f_{h,k} ]^{-1} \mathbf{R}^{-1} \mathbf{R} \nabla^2 f_{h,k} \Vert ,
\\
&=& 0.
\end{eqnarray*}
It is intuitive to consider that $\Vert \mathbf{I} - \mathbf{P} [\nabla_H^2 f_{h,k}]^{-1} \mathbf{R} \nabla^2 f_{h,k} \Vert$ should be small and less than $1$ when $\text{rank}(\mathbf{P})$ is close to but not equal to $N$. However, the above intuition is not true, and we prove this in the following lemma.
\begin{lemma} \label{lm: RankBD}
Suppose $\text{rank}(\mathbf{P}) \neq N$, then
\[
1 \leq \Vert \mathbf{I} - \mathbf{P} [\nabla_H^2 f_{h,k}]^{-1} \mathbf{R} \nabla^2 f_{h,k} \Vert \leq \sqrt{\frac{L_h}{\mu_h}},
\]
where $L_h$ and $\mu_h$ are defined in Assumption \ref{assm:NeMOLipAssm}. The operator $\nabla^2_H$ is defined in (\ref{eq:NeMOG-fH}).
\end{lemma}
\begin{proof}
Since $\nabla^2 f_{h,k}$ is a positive definite matrix, consider the  eigendecomposition of $\nabla^2 f_{h,k}$,
\[
\nabla^2 f_{h,k} = \mathbf{U} \mathbf{\Sigma} \mathbf{U}^T,
\]
where $\mathbf{\Sigma}$ is a diagonal matrix containing the eigenvalues of $\nabla^2 f_{h,k}$, and $\mathbf{U}$ is a orthogonal matrix where its columns are eigenvectors of $\nabla^2 f_{h,k}$. We then have
\begin{eqnarray*}
& &  \mathbf{I} - \mathbf{P} [\nabla_H^2 f_{h,k}]^{-1} \mathbf{R} \nabla^2 f_{h,k} 
\\
&=& \mathbf{I} - \mathbf{P} [\mathbf{R} \nabla^2 f_{h,k} \mathbf{P}]^{-1} \mathbf{R} \nabla^2 f_{h,k}  ,
\\
&=& \mathbf{U} \mathbf{\Sigma}^{-1/2} \mathbf{\Sigma}^{1/2} \mathbf{U}^T -  \mathbf{U} \mathbf{\Sigma}^{-1/2} \mathbf{\Sigma}^{1/2} \mathbf{U}^T \mathbf{P} [\mathbf{R} \mathbf{U} \mathbf{\Sigma}^{1/2} \mathbf{\Sigma}^{1/2} \mathbf{U}^T \mathbf{P}]^{-1}   \mathbf{R} \mathbf{U} \mathbf{\Sigma}^{1/2} \mathbf{\Sigma}^{1/2} \mathbf{U}^T  ,
\\
&=&  \mathbf{U} \mathbf{\Sigma}^{-1/2} \mathbf{\Sigma}^{1/2} \mathbf{U}^T 
\\
&& \: -  \mathbf{U} \mathbf{\Sigma}^{-1/2} (\mathbf{\Sigma}^{1/2} \mathbf{U}^T \mathbf{P}) [ (\mathbf{\Sigma}^{1/2} \mathbf{U}^T \mathbf{P})^T (\mathbf{\Sigma}^{1/2} \mathbf{U}^T \mathbf{P})]^{-1}   (\mathbf{\Sigma}^{1/2} \mathbf{U}^T \mathbf{P})^T \mathbf{\Sigma}^{1/2} \mathbf{U}^T  ,
\\
&=& \mathbf{U} \mathbf{\Sigma}^{-1/2} (\mathbf{I} - \mathbf{\Gamma}_{\mathbf{\Sigma}^{1/2} \mathbf{U}^T \mathbf{P}} ) \mathbf{\Sigma}^{1/2} \mathbf{U}^T  ,
\end{eqnarray*}
where $\mathbf{\Gamma}_{\mathbf{\Sigma}^{1/2} \mathbf{U}^T \mathbf{P}}$ is the orthogonal projection operator onto the range of $\mathbf{\Sigma}^{1/2} \mathbf{U}^T \mathbf{P}$, and so
\begin{eqnarray*}
\Vert \mathbf{I} - \mathbf{P} [\nabla_H^2 f_{h,k}]^{-1} \mathbf{R} \nabla^2 f_{h,k} \Vert &=& \Vert \mathbf{U} \mathbf{\Sigma}^{-1/2} (\mathbf{I} - \mathbf{\Gamma}_{\mathbf{\Sigma}^{1/2} \mathbf{U}^T \mathbf{P}} ) \mathbf{\Sigma}^{1/2} \mathbf{U}^T  \Vert ,
\\
&=& \Vert  \mathbf{\Sigma}^{-1/2} (\mathbf{I} - \mathbf{\Gamma}_{\mathbf{\Sigma}^{1/2} \mathbf{U}^T \mathbf{P}} ) \mathbf{\Sigma}^{1/2} \Vert .
\end{eqnarray*}
For the upper bound, we have
\[
\Vert  \mathbf{\Sigma}^{-1/2} (\mathbf{I} - \mathbf{\Gamma}_{\mathbf{\Sigma}^{1/2} \mathbf{U}^T \mathbf{P}} ) \mathbf{\Sigma}^{1/2} \Vert \leq \Vert  \mathbf{\Sigma}^{-1/2}  \Vert \Vert  (\mathbf{I} - \mathbf{\Gamma}_{\mathbf{\Sigma}^{1/2} \mathbf{U}^T \mathbf{P}} ) \Vert \Vert  \mathbf{\Sigma}^{1/2} \Vert \leq \sqrt{\frac{L_h}{\mu_h}},
\]
since $\mathbf{I} - \mathbf{\Gamma}_{\mathbf{\Sigma}^{1/2} \mathbf{U}^T \mathbf{P}}$ is an orthogonal projector and $\Vert  (\mathbf{I} - \mathbf{\Gamma}_{\mathbf{\Sigma}^{1/2} \mathbf{U}^T \mathbf{P}} ) \Vert \leq 1$. For the lower bound, we have
\begin{eqnarray*}
 \Vert  \mathbf{\Sigma}^{-1/2} (\mathbf{I} - \mathbf{\Gamma}_{\mathbf{\Sigma}^{1/2} \mathbf{U}^T \mathbf{P}} ) \mathbf{\Sigma}^{1/2} \Vert 
&=& \Vert  \mathbf{\Sigma}^{-1/2} (\mathbf{I} - \mathbf{\Gamma}_{\mathbf{\Sigma}^{1/2} \mathbf{U}^T \mathbf{P}} )(\mathbf{I} - \mathbf{\Gamma}_{\mathbf{\Sigma}^{1/2} \mathbf{U}^T \mathbf{P}} ) \mathbf{\Sigma}^{1/2} \Vert ,
\\
&=& \Vert  \mathbf{\Sigma}^{-1/2} (\mathbf{I} - \mathbf{\Gamma}_{\mathbf{\Sigma}^{1/2} \mathbf{U}^T \mathbf{P}} )\mathbf{\Sigma}^{1/2}\mathbf{\Sigma}^{-1/2}(\mathbf{I} - \mathbf{\Gamma}_{\mathbf{\Sigma}^{1/2} \mathbf{U}^T \mathbf{P}} ) \mathbf{\Sigma}^{1/2} \Vert ,
\\
&\leq & \Vert  \mathbf{\Sigma}^{-1/2} (\mathbf{I} - \mathbf{\Gamma}_{\mathbf{\Sigma}^{1/2} \mathbf{U}^T \mathbf{P}} )\mathbf{\Sigma}^{1/2}\Vert  \Vert \mathbf{\Sigma}^{-1/2}(\mathbf{I} - \mathbf{\Gamma}_{\mathbf{\Sigma}^{1/2} \mathbf{U}^T \mathbf{P}} ) \mathbf{\Sigma}^{1/2} \Vert ,
\\
& = & \Vert  \mathbf{\Sigma}^{-1/2} (\mathbf{I} - \mathbf{\Gamma}_{\mathbf{\Sigma}^{1/2} \mathbf{U}^T \mathbf{P}} )\mathbf{\Sigma}^{1/2}\Vert^2 .
\end{eqnarray*}
The assumption $\text{rank}(\mathbf{P}) \neq N$ implies 
\[
\mathbf{I} \neq \mathbf{\Gamma}_{\mathbf{\Sigma}^{1/2} \mathbf{U}^T \mathbf{P}} \quad\text{and}\quad \Vert \mathbf{\Sigma}^{-1/2}(\mathbf{I} - \mathbf{\Gamma}_{\mathbf{\Sigma}^{1/2} \mathbf{U}^T \mathbf{P}} ) \mathbf{\Sigma}^{1/2} \Vert \neq 0.
\]
Therefore,
$
1 \leq \Vert \mathbf{\Sigma}^{-1/2}(\mathbf{I} - \mathbf{\Gamma}_{\mathbf{\Sigma}^{1/2} \mathbf{U}^T \mathbf{P}} ) \mathbf{\Sigma}^{1/2} \Vert ,
$
as required.
\end{proof}
Lemma \ref{lm: RankBD} clarifies the fact that the term $\Vert \mathbf{I} - \mathbf{P} [\nabla_H^2 f_{h,k}]^{-1} \mathbf{R} \nabla^2 f_{h,k} \Vert$ is at least $1$ when $n < N$. This fact reduces the usefulness of the composite convergence rate in Theorem \ref{thm: composConv}. In Section \ref{sec:PDE}, we will investigate {the term} $\Vert(\mathbf{I} - \mathbf{PR})(\mathbf{x}_{h,k} - \mathbf{x}_{h,\star}) \Vert$ and show that it is sufficiently small in a specific case.

\section{PDE-based Problems: One-dimensional Case}\label{sec:PDE}

In this section, we study the Newton-type multilevel model that arises from PDE-based problems. We begin with introducing the basic setting, and then we analyze the coarse correction step in this specific case. Building upon the composite rate in Section \ref{sec:NeMOCompostRateIntro}, at the end of this section we re-derive the composite rate with a more insightful bound of $\Vert(\mathbf{I} - \mathbf{PR})(\mathbf{x}_{h,k} - \mathbf{x}_{h,\star})\Vert$. As mentioned in Section \ref{sec:NeMOconvAnalysis}, this quantity is critical in analyzing the performance and complexity of NeMO.

For the simplicity of the analysis, we consider specifically the one-dimensional case, i.e. the decision variable of the infinite dimensional problems is a functional in $\mathbb{R}$. We further assume that the decision variable is discretized uniformly over $[0,1]$ with value $0$ on the boundary. We would like to clarify that the approach of analysis in this section could be applied to more general and high-dimensional settings.

\subsection{Newton-type Multilevel Model by One-dimensional Interpolations}

\begin{figure}[t]
\centering
\begin{tikzpicture}[thick,scale=0.88, every node/.style={scale=0.88}]
\draw [thick] plot[smooth, tension=.7] coordinates {(-10,-0) (0,-0)};

\draw      [thick, fill = black] (-10,-0) ellipse (0.2 and 0.2) ;
\draw      [thick] (-8.75,0.925) ellipse (0.2 and 0.2);
\draw      [thick, fill = black] (-7.5,1.85) ellipse (0.2 and 0.2);
\draw      [thick] (-6.25,0.925) ellipse (0.2 and 0.2);
\draw      [thick, fill = black] (-5,-0) ellipse (0.2 and 0.2);
\draw      [thick] (-3.75,-0.925) ellipse (0.2 and 0.2);
\draw      [thick, fill = black] (-2.5,-1.85) ellipse (0.2 and 0.2);
\draw      [thick] (-1.25,-0.925) ellipse (0.2 and 0.2);
\draw      [thick, fill = black]  (-0,-0) ellipse (0.2 and 0.2);

\node (v1) at (-10,-0)  {};
\node (v3) at(-7.5,1.85) {};
\node (v5) at (-5,-0) {};
\node (v7) at (-2.5,-1.85) {};
\node (v9) at  (-0,-0) {};

\draw [densely dashed] (v1) -- (v3)-- (v5)-- (v7)-- (v9) ;
\end{tikzpicture}
\caption{$\mathbf{P}$ in (\ref{eq: PDEP})} \label{fig:PDEp}
\end{figure}
\begin{figure}[t]
\centering
\begin{tikzpicture}[thick,scale=0.88, every node/.style={scale=0.88}]
\draw [thick] plot[smooth, tension=.7] coordinates {(-10,-0) (0,-0)};

\draw        [thick, fill = black] (-10,-0) ellipse (0.2 and 0.2) ;
\draw        [thick, fill = black] (-8.75,1.5) ellipse (0.2 and 0.2);
\draw        [thick, fill = black] (-7.5,1.85) ellipse (0.2 and 0.2);
\draw        [thick, fill = black] (-6.25,1.5) ellipse (0.2 and 0.2);
\draw        [thick, fill = black] (-5,-0) ellipse (0.2 and 0.2);
\draw        [thick, fill = black] (-3.75,-1.5) ellipse (0.2 and 0.2);
\draw        [thick, fill = black] (-2.5,-1.85) ellipse (0.2 and 0.2);
\draw        [thick, fill = black] (-1.25,-1.5) ellipse (0.2 and 0.2);
\draw        [thick, fill = black]  (-0,-0) ellipse (0.2 and 0.2);

\draw       [thick] (-7.5,1.7) ellipse (2.3 and 0.6);
\draw       [thick,rotate around={-50:(-5,-0)}] (-5,-0) ellipse (2.5 and 0.7);
\draw       [thick] (-2.5,-1.7) ellipse (2.3 and 0.6);

\node (v1) at (-10,-0)  {};
\node (v2) at (-8.75,1.5) {};
\node (v3) at(-7.5,1.85) {};
\node (v4) at (-6.25,1.5) {};
\node (v5) at (-5,-0) {};
\node (v6) at (-3.75,-1.5) {};
\node (v7) at (-2.5,-1.85) {};
\node (v8) at (-1.25,-1.5) {};
\node (v9) at  (-0,-0) {};

\draw   [densely dashed] (v1)-- (v2)-- (v3)-- (v4)-- (v5)-- (v6)-- (v7)-- (v8)-- (v9);
\end{tikzpicture}
\caption{$\mathbf{R}$ in (\ref{eq: PDER})} \label{fig:PDEr}
\end{figure}

For one dimensional problems, we consider the standard linear prolongation operator and restriction operator. Based on the traditional setting in multigrid research, we define the following Newton-type multilevel model.
\begin{itemize}
	\item $N$ is an even number,
	\item the (fine) discretized decision variable is in $\mathbf{R}^{N-1}$, and
	\item the coarse model is in $\mathbf{R}^{N/2 -1}$.
\end{itemize}
For interpolation operator $\mathbf{P}\in \mathbb{R}^{(N-1) \times (N/2 -1)}$, we consider 
\begin{equation}\label{eq: PDEP}
\mathbf{P} 
= \frac{1}{2}
\begin{pmatrix}
  1 &   &        &    \\
  2 &   &        &    \\
  1 & 1 &        &    \\
    & 2 &        &    \\
    & 1 &        &    \\
    &   & \ddots &  1 \\
    &   &        &  2 \\
    &   &        &  1 \\   
\end{pmatrix},
\end{equation}
and the restriction operator 
\begin{equation}\label{eq: PDER}
\mathbf{R} = \dfrac{1}{2} \mathbf{P}^T \in \mathbb{R}^{(N/2 -1) \times (N-1)  } .
\end{equation}
Notice that the $\mathbf{P}$ and $\mathbf{R}$ in (\ref{eq: PDEP}) and (\ref{eq: PDER})  have geometric meanings, and they are one of the standard pairs of operators in multilevel and multigrid methods \cite{Briggs2000}. As shown in Figure \ref{fig:PDEp}, $\mathbf{P}$ is an interpolation operator such that one point is interpolated linearly between every two points. On the other hand, from Figure \ref{fig:PDEr}, $\mathbf{R}$ performs restriction by doing weighted average onto every three points. These two operators assume the boundary condition is zero for both end points. We emphasize that the approach of convergence analysis in this section is not restricted for this specific pair of $\mathbf{P}$ and $\mathbf{R}$. We believe the general approach could be applied to interpolation type operators, especially operators that are designed for PDE-based optimization problems.

\subsection{Analysis}
With the definitions (\ref{eq: PDEP}) and (\ref{eq: PDER}), we investigate the convergence behaviour of the coarse correction step. The analytical tool we used in this section is Taylor's expansion. To deploy this technique, we consider interpolations over the elements of vectors. In particular, we consider interpolations that are twice differentiable with the following definition.
\begin{definition}\label{def:NeMOPDEInterFunClass}
For any vector $\mathbf{r} \in \mathbb{R}^{N-1}$, we denote $\mathcal{F}^{N-1}_\mathbf{r}$ to be the set of twice differentiable functions such that $\forall w \in \mathcal{F}^{N-1}_\mathbf{r}$,
\[
w(0) = w(1) = 0, \quad and \quad w_i = w(y_i) = (\mathbf{r})_i ,
\]
where $y_i = i/N$ for $i = 1,2,\dots, N-1$.
\end{definition}

Using the definitions (\ref{eq: PDEP}) and (\ref{eq: PDER}), we can estimate the ``information loss'' via interpolations using the following proposition.
\begin{proposition}\label{prop:pdePRvDef}
Suppose $\mathbf{P}$ and $\mathbf{R}$ are defined in (\ref{eq: PDEP}) and (\ref{eq: PDER}), respectively. For any vector $\mathbf{r}_h \in \mathbb{R}^{N-1}$, we denote $(\mathbf{r}_{h})_{0}=(\mathbf{r}_{h})_{N}=0$ and obtain
\begin{eqnarray*}
(\mathbf{P} \mathbf{R} \mathbf{r}_{h})_j = 
\begin{cases}
\frac{1}{4} ( (\mathbf{r}_{h})_{j-1}+2 (\mathbf{r}_{h})_{j}+ (\mathbf{r}_{h})_{j+1}) \quad \text{if $j$ is even,} 
\\
\frac{1}{8} ( (\mathbf{r}_{h})_{j-2}+2(\mathbf{r}_{h})_{j-1}+2(\mathbf{r}_{h})_{j}+2(\mathbf{r}_{h})_{j+1} + (\mathbf{r}_{h})_{j+2}) \quad \text{if $j$ is odd},
\end{cases}
\end{eqnarray*}
for $j=1,2,\dots,N-1$.
\end{proposition}
\begin{proof}
By the definition of $\mathbf{R}$ and $\mathbf{P}$, we have
\begin{equation*}
(\mathbf{R} \mathbf{r}_h)_{j} = \frac{1}{4} ( (\mathbf{r}_{h})_{2j-1}+2(\mathbf{r}_{h})_{2j}+ (\mathbf{r}_{h})_{2j+1}) , \quad 1 \leq j \leq \frac{n}{2} -1.
\end{equation*}
So
\begin{equation*}
(\mathbf{P} \mathbf{R} \mathbf{r}_h)_{j} = (\mathbf{R} \mathbf{r}_h)_{j/2} = \frac{1}{4} ( (\mathbf{r}_{h})_{j-1}+2 (\mathbf{r}_{h})_{j}+ (\mathbf{r}_{h})_{j+1}) \quad \text{if $j$ is even,} 
\end{equation*}
and
\begin{eqnarray*}
(\mathbf{P} \mathbf{R} \mathbf{r}_h)_j &=& \frac{1}{2} \left( (\mathbf{R} \mathbf{r}_h)_{(j-1)/2} + (\mathbf{R} \mathbf{r}_h)_{(j+1)/2} \right) ,
\\
&=& \frac{1}{8} ( (\mathbf{r}_{h})_{j-2}+2(\mathbf{r}_{h})_{j-1}+2(\mathbf{r}_{h})_{j}+2(\mathbf{r}_{h})_{j+1} + (\mathbf{r}_{h})_{j+2}) \quad \text{if $j$ is odd}.
\end{eqnarray*}
So we obtain the desired result.
\end{proof}
Using the above proposition and Taylor's expansion, we obtain the following lemma.
\begin{lemma}\label{lm: PDEerrBD1}
Suppose $\mathbf{P}$ and $\mathbf{R}$ are defined in (\ref{eq: PDEP}) and (\ref{eq: PDER}), respectively. For any vector $\mathbf{r}_h \in \mathbb{R}^{N-1}$,
\[
\Vert (\mathbf{I} - \mathbf{P} \mathbf{R} )\mathbf{r}_{h} \Vert_{\infty} \leq \min_{w \in \mathcal{F}^{N-1}_{\mathbf{r}_h}} \max_{y \in [0,1]}\vert w''(y) \vert \frac{3}{4N^2} .
\]
Note that the definition of $\mathcal{F}^{N-1}_{\mathbf{r}_h}$ follows from Definition \ref{def:NeMOPDEInterFunClass}.
\end{lemma}
\begin{proof}
Using Proposition \ref{prop:pdePRvDef} and Taylor's Theorem, in the case that $j$ is even, we obtain
\begin{eqnarray*}
\frac{1}{4} ( (\mathbf{r}_{h})_{j-1}+2 (\mathbf{r}_{h})_{j}+ (\mathbf{r}_{h})_{j+1})  
&=& 
\frac{1}{4} \left( w\left(y_{j-1}\right)+2w\left(y_{j} \right)+w\left(y_{j+1}\right) \right),
\\
 &=& w\left(y_{j}\right)+ \frac{w''(y_{c1})}{8} \frac{1}{N^2} + \frac{w''(y_{c2})}{8} \frac{1}{N^2},
 \\
 &=& (\mathbf{r}_{h})_{j} + \frac{w''(y_{c1})+w''(y_{c2})} {8} \frac{1}{N^2} ,
\end{eqnarray*}
where $w(\cdot) \in \mathcal{F}^{N-1}_{\mathbf{r}_h}$, $y_{j-1} \leq y_{c1} \leq y_{j}$, and $y_j \leq y_{c2} \leq y_{j+1}$. Similarly, in the case that $j$ is odd, we have 
\begin{multline}
\frac{1}{8} ( (\mathbf{r}_{h})_{j-2}+2(\mathbf{r}_{h})_{j-1}+2(\mathbf{r}_{h})_{j}+2(\mathbf{r}_{h})_{j+1} + (\mathbf{r}_{h})_{j+2}) 
\\
= (\mathbf{r}_{h})_{j} + \frac{4w''(y_{c3})+2 w''(y_{c4})+ 2w''(y_{c5})+4w''(y_{c6})} {16} \frac{1}{N^2},
\end{multline}
where $y_{j-2} \leq y_{c3} \leq y_j$, $y_{j-1} \leq y_{c4} \leq y_{j}$, $y_j \leq y_{c5} \leq y_{j+1}$, and $y_j \leq y_{c6} \leq y_{j+2}$. Therefore,
\[
\Vert (\mathbf{I} - \mathbf{P} \mathbf{R} )\mathbf{r}_{h} \Vert_{\infty} \leq \max_{y\in [0,1]}\vert w''(y) \vert \frac{3}{4N^2} \quad \text{for} \quad \forall w(\cdot) \in \mathcal{F}^{N-1}_{\mathbf{r}_h}.
\]
So we obtain the desired result.
\end{proof}
Lemma \ref{lm: PDEerrBD1} provides upper bound of $\Vert (\mathbf{I} - \mathbf{P} \mathbf{R} )\mathbf{r}_{h} \Vert_{\infty}$, for any $\mathbf{r}_{h} \in \mathbf{R}^{N-1}$. This result can be used to derive the upper bound of $\Vert (\mathbf{I} - \mathbf{P} \mathbf{R} )(\mathbf{x}_{h,k}-\mathbf{x}_{h,\star }) \Vert$, where $\mathbf{r}_{h} = \mathbf{x}_{h,k}-\mathbf{x}_{h,\star }$. 
As we can see, if $\vert w''(y) \vert = \mathcal{O} (1)$, where $w \in \mathcal{F}^{N-1}_{\mathbf{r}_h}$, then $\Vert (\mathbf{I} - \mathbf{P} \mathbf{R} )\mathbf{r}_{h} \Vert_{\infty} = \mathcal{O} (N^{-2})$. This can be explained by the fact that when the mesh size is fine enough (i.e. large $N$), linear interpolation and restriction provide very good estimations of the fine model.

In the following lemma, we provide an upper bound of $\vert w'' \vert$ in terms of the original vector $\mathbf{r}_h$. The idea is to specify the interpolation method in which we construct $w$, and we will use cubic spline in particular. Cubic spline is one of the standard interpolation methods, and the output interpolated function $w$ satisfies the setting in Definition \ref{def:NeMOPDEInterFunClass} and Lemma \ref{lm: PDEerrBD1}.

\begin{lemma}\label{lm: PDEerrBD2}
Suppose $\mathbf{P}$ and $\mathbf{R}$ are defined in (\ref{eq: PDEP}) and (\ref{eq: PDER}), respectively. For any vector $\mathbf{r}_h \in \mathbb{R}^{N-1}$, we obtain
\[
\Vert (\mathbf{I} - \mathbf{P} \mathbf{R} )\mathbf{r}_{h} \Vert_{\infty}  \leq  \frac{9}{4N^2}\Vert \mathbf{A} \mathbf{r}_h \Vert_{\infty},
\]
where
\[
\mathbf{A} = N^2 
\begin{pmatrix}
  2 & -1 &  & & \\
  -1 & 2 & -1 & &  \\
    & -1 & \ddots & \ddots &  \\
   &  & \ddots & 2& -1  \\
   &  &  & -1 & 2 
\end{pmatrix}.
\]
\end{lemma}
\begin{proof}
We follow the notation in Definition \ref{def:NeMOPDEInterFunClass}. For $w \in \mathcal{F}^{N-1}_{\mathbf{r}_h}$ that is constructed via cubic spline, in the interval $\left(y_i , y_{i+1} \right)$, we have
\[
w(y) = A w_i + B w_{i+1} + C w_{i}'' + D w_{i+1}'' ,
\]
where 
\begin{eqnarray*}
A &=& \frac{y_{i+1}-y}{y_{i+1}-y_{i}},
\\
B &=& \frac{y-y_{i}}{y_{i+1} - y_{i} },
\\
C &=& \frac{1}{6} (A^3 - A) (y_{i+1}-y_{i})^2,
\\
D &=& \frac{1}{6} (B^3 - B) (y_{i+1}-y_{i})^2.
\end{eqnarray*}
It is known from \cite{Press1996} that
\begin{equation}\label{eq:CubSplineD}
\frac{\text{d}^2 w}{\text{d}y^2} = A w_i'' + B w_{i+1}'',
\end{equation}
and
\begin{equation}\label{eq:CubSplineRel}
\frac{y_i - y_{i-1}}{6} w_{i-1}'' + \frac{y_{i+1} - y_{i-1}}{3} w_{i}'' + \frac{y_{i+1} - y_{i}}{6} w_{i+1}'' = \frac{w_{i+1} - w_i}{y_{i+1} - y_{i}} - \frac{w_{i} - w_{i-1}}{y_{i} - y_{i-1}},
\end{equation}
and for $i = 1,2,\dots , N-1$. Using the above equation (\ref{eq:CubSplineD}), at the interval $\left(y_i,y_{i+1}\right)$, we obtain
\begin{eqnarray*}
\Bigg\vert \frac{\text{d}^2 w}{\text{d}y^2} \Bigg\vert = \big\vert A w_i'' + B w_{i+1}'' \big\vert &=& \Bigg\vert \frac{y_{i+1}-y}{y_{i+1}-y_{i}} w_i'' + \frac{y-y_{i}}{y_{i+1} - y_{i} } w_{i+1}'' \Bigg\vert , 
\\
&\leq & \Bigg\vert \frac{y_{i+1}-y}{y_{i+1}-y_{i}}\Bigg\vert \vert w_i'' \vert + \Bigg\vert \frac{y-y_{i}}{y_{i+1} - y_{i} }\Bigg\vert \vert w_{i+1}'' \vert , 
\\
&\leq & \max \{ \vert w_{i}''\vert , \vert w_{i+1}''\vert \} .
\end{eqnarray*}
Suppose $j \in \arg \max_i \{ \vert w_{i}''\vert \}_i$, then from (\ref{eq:CubSplineRel}) and the fact that $y_{j+1} - y_{j} = 1/N$,
\begin{eqnarray*}
  \frac{y_{j+1} - y_{j-1}}{3} w_{j}''  &=& \frac{w_{j+1} - w_j}{y_{j+1} - y_{j}} - \frac{w_{j} - w_{j-1}}{y_{j} - y_{j-1}} - \frac{y_j - y_{j-1}}{6} w_{j-1}'' - \frac{y_{j+1} - y_{j}}{6} w_{j+1}'' ,
  \\
\frac{2}{3N} w_{j}''  &=& N( w_{j+1} - w_j) - N(w_{j} - w_{j-1}) - \frac{1}{6N} w_{j-1}'' - \frac{1}{6N} w_{j+1}'' ,
  \\
2 w_{j}''  &=& 3 N^2 ( w_{j+1} - 2 w_j  + w_{j-1}) - \frac{1}{2} w_{j-1}'' - \frac{1}{2} w_{j+1}'' .
\end{eqnarray*}
Thus,
\begin{eqnarray*}
 \vert 2 w_{j}'' \vert & \leq & 3 N^2 \vert  w_{j+1} - 2 w_j  + w_{j-1} \vert  +   \frac{1}{2} \vert w_{j-1}''\vert +   \frac{1}{2} \vert  w_{j+1}'' \vert ,
 \\
 2\vert  w_{j}'' \vert & \leq & 3 N^2 \vert  w_{j+1} - 2 w_j  + w_{j-1} \vert  +   \frac{1}{2} \vert w_{j}''\vert +   \frac{1}{2} \vert  w_{j}'' \vert ,
  \\
 \vert  w_{j}'' \vert & \leq & 3 N^2 \vert  w_{j+1} - 2 w_j  + w_{j-1} \vert  .
\end{eqnarray*}
Therefore,
\[
\vert  w_{i}'' \vert  \leq \max_{i} 3 N^2 \vert  w_{i+1} - 2 w_i  + w_{i-1} \vert ,
\]
and so,
\[
\Vert (\mathbf{I} - \mathbf{P} \mathbf{R} )\mathbf{r}_{h} \Vert_{\infty} \leq \max_{y\in [0,1]}\vert w''(y) \vert \frac{3}{4N^2} \leq  \max_{i} \frac{9\vert  w_{i+1} - 2 w_i  + w_{i-1} \vert}{4} = \frac{9}{4 N^2}\Vert \mathbf{A} \mathbf{r}_h \Vert_{\infty},
\]
as required.
\end{proof}
Lemma \ref{lm: PDEerrBD2} provides the discrete version of the result presented in Lemma \ref{lm: PDEerrBD1}. The matrix $\mathbf{A}$ is the discretized Laplacian operator, which is equivalent to twice differentiation using finite difference with a uniform mesh.

\subsection{Convergence}
With all the results, we revisit the composite convergence rate with the following Corollary.

\begin{corollary}\label{col:PDEconv}
Suppose $\mathbf{P}$ and $\mathbf{R}$ are defined in (\ref{eq: PDEP}) and (\ref{eq: PDER}), respectively. If the coarse correction step $\hat{\mathbf{d}}_{h,k}$ in (\ref{eq:NeMOStpDef}) is taken with $\alpha_{h,k}=1$, then
\begin{eqnarray*}
\Vert \mathbf{x}_{h,k+1} - \mathbf{x}_{h,\star} \Vert 
& \leq & \sqrt{\frac{L_h}{\mu_h}}  \min_{w \in \mathcal{F}^{N-1}_{\mathbf{x}_{h,k} - \mathbf{x}_{h,\star}}} \max_{y\in [0,1]}\vert w''(y) \vert \frac{3}{4N^{3/2}}  +  \frac{M_h \omega^2 \xi^2}{2\mu_h} \Vert \mathbf{x}_{h,k} - \mathbf{x}_{h,\star} \Vert^2 ,
\\
& \leq &  \frac{9}{4N^{3/2}} \sqrt{\frac{L_h }{\mu_h}}   \Vert \mathbf{A} (\mathbf{x}_{h,k} - \mathbf{x}_{h,\star}) \Vert +  \frac{M_h \omega^2 \xi^2}{2\mu_h} \Vert \mathbf{x}_{h,k} - \mathbf{x}_{h,\star} \Vert^2 ,
\end{eqnarray*}
where $\mathbf{A}$ is defined in Lemma \ref{lm: PDEerrBD2}. Note that $M_h$, $L_h$, and $\mu_h$ are defined in Assumption~\ref{assm:NeMOLipAssm}, $\omega = \max \{ \Vert \mathbf{P} \Vert , \Vert \mathbf{R} \Vert \}$, and $\xi = \Vert \mathbf{P}^+ \Vert$.
\end{corollary}
\begin{proof}
\begin{eqnarray*}
\Vert \mathbf{x}_{h,k+1} - \mathbf{x}_{h,\star} \Vert 
& \leq & \Vert \mathbf{I} - \mathbf{P} [\nabla_H^2 f_{h,k}]^{-1} \mathbf{R} \nabla^2 f_{h,k} \Vert  \Vert (\mathbf{I} - \mathbf{PR})(\mathbf{x}_{h,k} - \mathbf{x}_{h,\star}) \Vert 
\\
&& \: +  \frac{M_h \omega^2 \xi^2}{2\mu_h} \Vert \mathbf{x}_{h,k} - \mathbf{x}_{h,\star} \Vert^2 ,
\\
& \leq & \sqrt{\frac{L_h}{\mu_h}}  \min_{w \in \mathcal{F}^{N-1}_{\mathbf{x}_{h,k} - \mathbf{x}_{h,\star}}} \max_{y\in [0,1]}\vert w''(y) \vert \frac{3}{4N^{3/2}}  +  \frac{M_h \omega^2 \xi^2}{2\mu_h} \Vert \mathbf{x}_{h,k} - \mathbf{x}_{h,\star} \Vert^2 ,
\\
& \leq &  \frac{9}{4N^{3/2}} \sqrt{\frac{L_h}{\mu_h}}   \Vert \mathbf{A} (\mathbf{x}_{h,k} - \mathbf{x}_{h,\star}) \Vert +  \frac{M_h \omega^2 \xi^2}{2\mu_h} \Vert \mathbf{x}_{h,k} - \mathbf{x}_{h,\star} \Vert^2 ,
\end{eqnarray*}
as required.
\end{proof}
Corollary \ref{col:PDEconv} provides the convergence of using Newton-type multilevel model for PDE-based problems that we considered. This result shows the complementary of fine correction step and coarse correction step. Suppose the fine correction step can effectively reduce $\Vert \mathbf{A} (\mathbf{x}_{h,k} - \mathbf{x}_{h,\star}) \Vert$, then the coarse correction step could yield major reduction based on the result shown in Corollary \ref{col:PDEconv}.

\section{Numerical Experiments}\label{sec:NeMONumExp}
In this section, we verify our convergence results with a numerical example. This example satisfies the assumptions of Section \ref{sec:PDE}, and it is an one-dimensional Poisson's equation, which is a standard example in numerical analysis and multigrid algorithms. In the second part of this section, we will compare NeMO with other algorithms. 

\subsection{Poisson's Equation}

We consider an one-dimensional Poisson's equation
\[
- \frac{\text{d}^2}{\text{d}q^2} u = w(q) \quad \text{in} \ \ [0,1], \quad u(0)=u(1)=0,
\]
where $w(q)$ is chosen as
\[
w(q) = \sin(4\pi q) + 8\sin(32\pi q)  + 16\sin(64\pi q).
\]
We discretize the above problem and denote $\mathbf{x},\mathbf{b}\in \mathbb{R}^{N-1}$ , where $(\mathbf{x})_i = u(i/N)$ and $(\mathbf{b})_i = w(i/N)$, for $i=1,2,\dots, N-1$. By using finite difference, we approximate the above equation with
\begin{equation}\label{eq:PossionEq}
\min_{\mathbf{x}\in \mathbb{R}^{N-1}}\frac{1}{2}\mathbf{x}^T\mathbf{A}\mathbf{x} - \mathbf{b}^T\mathbf{x},
\end{equation}
where $\mathbf{A}$ is defined as in Lemma \ref{lm: PDEerrBD2}, which is a discretized Laplacian operator.

\begin{figure*}[t]
\begin{center}
\begin{minipage}[t]{0.48\linewidth}
    \includegraphics[width=\linewidth]{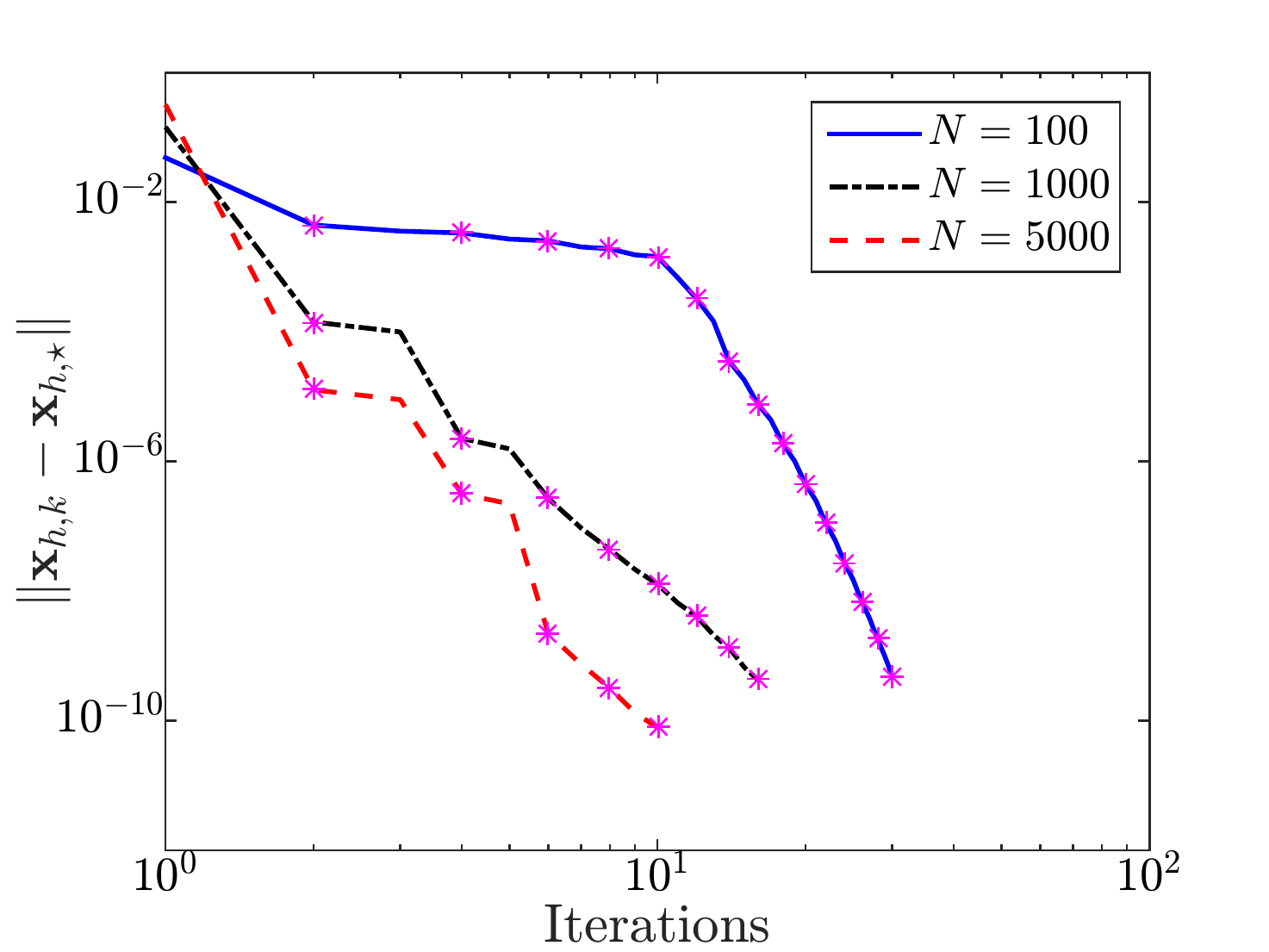}
    \caption{Convergence of solving Poisson's equation with different $N$'s}
    \label{fig:PossionResult}
\end{minipage}%
\ \ \ \ 
\begin{minipage}[t]{0.48\linewidth}
	\includegraphics[width=\linewidth]{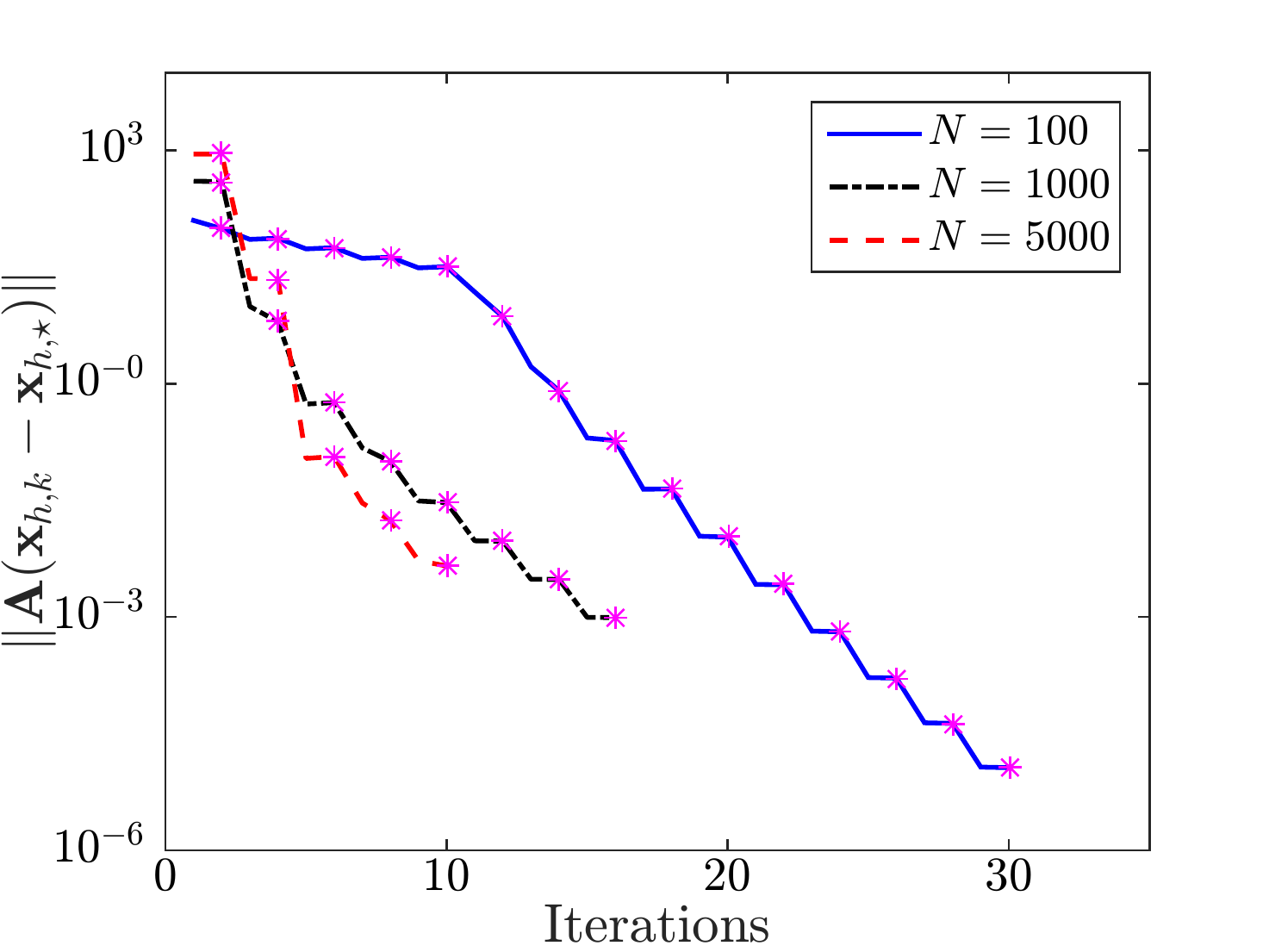}
    \caption{The smoothing effect with different $N$'s}
    \label{fig:PossionDiff}
\end{minipage} 
\end{center}
\end{figure*}

Figure \ref{fig:PossionResult} shows the convergence results of solving (\ref{eq:PossionEq}) with different $N$'s. In this example we use the prolongation and restriction operators that are defined in (\ref{eq: PDEP}) and (\ref{eq: PDER}). Steepest descent is used to compute the fine correction step. The pink stars in Figure \ref{fig:PossionResult} and Figure \ref{fig:PossionDiff} indicate where coarse correction steps were used.

As expected from Corollary \ref{col:PDEconv}, the performance of convergence is inversely proportional to the discretization level $N$. More interestingly, one can see the complementary of fine correction step and coarse correction step. From Figure \ref{fig:PossionResult}, fine correction steps are often deployed after coarse correction steps. Each pair of fine and coarse correction steps provides significant improvement in convergence. Figure \ref{fig:PossionDiff} shows the smoothing effect of the fine correction step by looking at the quantity $\Vert \mathbf{A} (\mathbf{x}_{h,k} - \mathbf{x}_{h,\star}) \Vert$, where $\mathbf{A}$ is the discretized Laplacian operator, as defined in Lemma \ref{lm: PDEerrBD2}. As opposed to coarse correction steps, fine correction steps are effective in reducing $\Vert \mathbf{A} (\mathbf{x}_{h,k} - \mathbf{x}_{h,\star}) \Vert$. Once the error is smoothed, coarse correction steps provide large reduction in error, as shown in Figure \ref{fig:PossionResult}.

\subsection{Numerical Performance}

Algorithm~\ref{alg: MLGM} offers great flexibility with respect to the choice of its various components, such as the interpolation operator, fine-level smoother, linear solver, etc. In our numerical experiments, we have used two variants:
\begin{enumerate}
	\item[A1.1.] The fine-level smoother is the damped Newton method with Armijo line search. Linear systems
	$$
	\mathbf{H}_h \mathbf{d} = -\mathbf{g}_h
	$$
	arising in the Newton method are solved by a direct solver, namely, by the Matlab's backslash operator.
	\item[A1.2.] The smoother is the Newton method as in A1.1.  However, assuming that we have an interpolation and a restriction operators $\mathbf{P}$ and $\mathbf{R}$ at our disposal, we can use it to solve the fine-level linear equation
	$$
	\mathbf{H}_h \mathbf{d} = -\mathbf{g}_h
	$$
	by a two-grid method with $\mathbf{H}_H = \mathbf{R} \mathbf{H}_h \mathbf{P}$.
\end{enumerate}
We will compare the above two methods with the MG/OPT algorithm \cite{Nash2000}
\begin{enumerate}	
	\item[A1.3.] As in A1.1 but with the coarse level matrix $\mathbf{H}_H$ being the exact Hessian of the coarse level problem.
	\item[A1.4.] As in A1.2 but with the coarse level matrix $\mathbf{H}_H$ being the exact Hessian of the coarse level problem.
\end{enumerate}
Further details common to all the above variants:
\begin{itemize}
\item Linear systems on the coarse level were solved by a direct method (Matlab backslahs operator).
\item Initial point: Set as in Matlab as
\begin{verbatim} 
rng('default');
x = 5.*randn(n,1);
\end{verbatim}
We did not use the obvious choice $\mathbf{x}=0$, as this is, for most examples, too close to the region of quadratic convergence of the Newton method. We wanted to see the effect of NeMO when most of the iterations lie outside this region.
\item Stopping tolerance: Assuming that we minimize a function $f$; Algorithm~\ref{alg: MLGM} has been stopped when $\|\nabla f(\mathbf{x})\|\leq \varepsilon_{\text{stop}}$, with $\varepsilon_{\text{stop}}=10^{-9}$ unless specified otherwise.
\item The control parameters $\kappa$ and $\varepsilon$ have been chosen as $\kappa= \displaystyle\frac{n_H}{n_h}$ and $\varepsilon=0.1$, unless specified otherwise. (Here $n_h$ and $n_H$ is the number of variables on the fine and the coarse level, respectively.)
\item The parameter of the standard Armijo line search is set to 0.01.
\item In Algorithms A1.2 and A1.4, the  fine level multigrid method was stopped as soon as the scaled residuum of the Newton equation was below 0.1.
\end{itemize}

In all examples, matrix $\mathbf{A}$ is the discretized two-dimensional Laplace operator. The discretization was performed on a square domain using the finite difference method and we considered homogeneous Dirichlet boundary conditions. When defining the levels, we started with an initial $3\times 3$ grid as ``level 1". Each next level used regular refinement doubling the number of discretization points in each coordinate. Hence ``level 2" corresponds to $5 \times 5$ and the corresponding matrix $\mathbf{A}\in\mathbb{R}^{9\times 9}$ (after elimination of the boundary points)
$$
\mathbf{A} = \frac{1}{3} 
\begin{pmatrix}
 8&     -1&            0&           -1&          -1&            0&            0&            0&            0   \\    
-1&      8&           -1&           -1&          -1&           -1&            0&            0&            0  \\     
 0&     -1&            8&            0&          -1&           -1&            0&            0&            0  \\     
-1&     -1&            0&            8&          -1&            0&           -1&           -1&            0  \\     
-1&     -1&           -1&           -1&           8&           -1&           -1&           -1&           -1  \\   
 0&     -1&           -1&            0&          -1&            8&            0&           -1&           -1  \\   
 0&      0&            0&           -1&          -1&            0&            8&           -1&            0  \\     
 0&      0&            0&           -1&          -1&           -1&           -1&            8&           -1  \\   
 0&      0&            0&            0&          -1&           -1&            0&           -1&            8
\end{pmatrix}.
$$
We use up to ten discterization levels with ``level 10" corresponding to a problem with $1050625 \times 1050625$ matrix $\mathbf{A}$, i.e., a problem with 1050625 variables.

The interpolation operators $\mathbf{P}=P_k^{k+1}$ from level $k$ to level $k+1$ are based on the nine-point interpolation scheme defined by the stencil $\begin{pmatrix}
\frac{1}{4}&\frac{1}{2}&\frac{1}{4}\\
\frac{1}{2}&1&\frac{1}{2}\\
\frac{1}{4}&\frac{1}{2}&\frac{1}{4}\\
\end{pmatrix}$. We use the full weighting restriction operators defined by $\mathbf{R}= \frac{1}{4}(P_k^{k+1})^T$; see, e.g., \cite{Hackbusch1985}. The interpolation operator between levels $k$ and $k+p$ is defined
by $\mathbf{P} = P_{k+p-1}^{k+p}\; P_{k+p-2}^{k+p-1}\cdots P_k^{k+1}$ and analogously for the restriction operator $\mathbf{R}$.

\paragraph*{Example 1}
Minimize the following function
$$
f(x) := \frac{1}{2}\mathbf{x}^T \mathbf{A} \mathbf{x} + h \lambda \sum_{i=1}^n (\mathbf{x}^2 e^\mathbf{x}-e^{\mathbf{x}}) - \mathbf{b}^T\mathbf{x}
$$
where $\lambda=10$ and $h=1/(n+1)$. Here
$\mathbf{A} $ is a matrix resulting from discretization of the Laplacian operator on a regular finite element mesh, using bilinear quadrilateral elements and $\mathbf{b}$ is the discretization of function
$$
{b}(x_1,x_2)=\left(9 \pi^2+e^{(x_1^2-x_1^3)\sin(3\pi x_2)}(x_1^2-x_1^3)+6x_1-2\right) \sin(3\pi x_1)
$$
on the same mesh. 

Table~\ref{tab:1} gives results obtained by NeMO variant A1.1 with a direct solver on all levels. In this (and the next) table the columns show the coarse level used (with 0 being the finest level); number of variables in the coarse level; total number of NeMO iterations; number of NeMO iteration on the fine level (i.e., number of times the fine level Newton equation has been solved); total CPU time on a MacBook Pro with 2.3 GHz Intel Core i5 processor running Matlab 2017b.

The first row of Table~\ref{tab:1} shows results with coarse level zero, i.e., for the standard damped Newton metod on the fine level. Hence we compare this line with the remaining NeMO results. Indeed, once we consider coarse level 2 and more, NeMO is substantially faster than the Newton method, in terms of the CPU time. For instance, for coarse level 2, we only have to visit the fine level in 5 iteration, the ``rest of the work" is performed on the coarse level. Figure~\ref{fig:101} shows the iteration history of NeMO with coarse level 2: most of the initial iteration are performed on the coarse level, while the final iterations are done on the fine level
\begin{table}[htbp]
	\centering
	\caption{Example 1, Algorithm A1.1 with direct solver on both levels}
	\begin{tabular}{crrrr}
		\toprule	
		coarse level &	coarse variables & total iter	&fine iter&	CPU time\\ 
		\midrule
		0&	1\,046\,529 & 20& 20&   88.1\\
		1&	     261\,121 & 23&	  6&	42.4\\
		2&	     65\,025 & 25&	 5&	   35.3\\
		3&	      16\,129 & 30&   6&	37.9\\
		4&	       3\,969 & 36&	  8&	47.1\\
		5&	           961 & 47&   11&	 60.6\\
		\bottomrule
	\end{tabular}%
	\label{tab:1}%
\end{table}%
\begin{figure}[h]
	\begin{center}
		\resizebox{0.7\hsize}{!}
		{\includegraphics{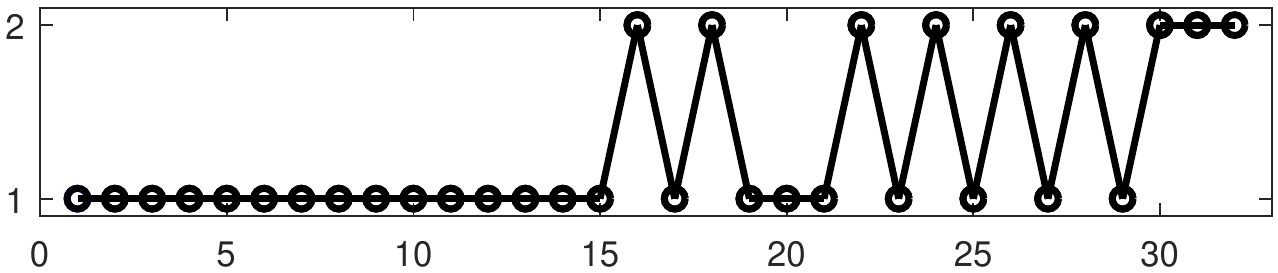}}
	\end{center}
	\caption{Example~1, levels visited in every iteration; 1 stands for the coarse level and 2 for the fine level.}\label{fig:101}
\end{figure}

Table~\ref{tab:1} confirms the advantage of NeMO as compared to the Newton method. 
However, assuming that we have an interpolation and a restriction operators $\mathbf{P}$ and $\mathbf{R}$ at our disposal, we can use it to solve the fine-level linear equation
$
\mathbf{H}_h \mathbf{d} = -\mathbf{g}_h
$
by a two-grid method with $\mathbf{H}_H = \mathbf{R} \mathbf{H}_h \mathbf{P}$. Table~\ref{tab:2} shows the result with this version of NeMO. In addition to the columns presented in Table~\ref{tab:1}, we also give the total number of two-grid iterations on the fine level (column ``mg iter"). As before, we first solve the problem using only the fine level (coarse level 0); the method then becomes equivalent to the standard nonlinear (Newton) multigrid method. The first three rows of Table~\ref{tab:2} show results with this method using coarse levels 1, 2 and 3 for the two-grid method. As we can see, this method is much more efficient then the Newton method with a direct solver (first row of Table~\ref{tab:1} ). In the next rows of Table~\ref{tab:2} we combine NeMO with the two-grid method for the linear equations on the fine level. As we can see, the advantage of NeMO to the nonlinear multigrid method is not so obvious in this case. NeMO with coarse level 3 is still the fastest method but only just.
\begin{table}[htbp]
	\centering
	\caption{Example 1, Algorithm A1.2 with two-grid solver on fine level}
	\begin{tabular}{ccrrrrr}
		\toprule	
		coarse level &	coarse level & coarse level & total iter & fine iter & mg iter&	CPU time\\ 
		for NeMO    & for mg          & variables\\   
		\midrule
		0&	1 & 1\,046\,529 & 20& 20& 20&  37.5\\
		0&	2 & 1\,046\,529 & 20& 20& 22&  27.9\\
		0&	3 & 1\,046\,529 & 21& 21&  33& 29.8\\
		1&	2 &     261\,121 & 25&	  8& 26&	43.6\\
		2&	2 &      65\,025 & 31&	 10& 19&   31.8\\
		3&	2 &       16\,129 & 30&   9& 11&	26.2\\
		4&	2 &        3\,969 & 33&	 10& 12&	28.3\\
		5&	2 &            961 & 48&   12& 14&	 36.8\\
		\bottomrule
	\end{tabular}%
	\label{tab:2}%
\end{table}%

Finally, to have a complete overview, we give in Table~\ref{tab:3} results for the MG/OPT method \cite{Nash2000,Wen2009} when the coarse level matrix for the linear system is computed as an exact Hessian of the objective function discretized on the coarse level. Again, the fine level linear system is either solved by a direct method (first part of Table~\ref{tab:3}) or by the two-grid method as above. One can see that using the two-grid solver would be slightly beneficial when the number of coarse level variables is small.
\begin{table}[htbp]
	\centering
	\caption{Example 1, two-level MG/OPT with direct solver on the fine level (Algorithm A1.3) and with a two-grid solver on the fine level (Algorithm A1.4)}
	\begin{tabular}{ccrrrrr}
		\toprule	
		fine level &	coarse level & coarse level & total iter & fine iter & mg iter&	CPU time\\ 
		solver     & for MG/OPT         & variables\\   
		\midrule
		direct&  1 &     261\,121 & 14&	  5& --&	30.3\\
		         &  2 &      65\,025 & 21&	 6&-- &   32.7\\
		         &	3 &       16\,129 & 28&   7&-- &	37.8\\
		         &	4 &        3\,969 & 33&	 9&-- &	47.1\\
		         &  5 &	            961 & 28&   12&-- &	 51.6\\
		\midrule
		mg&	  1 &     261\,121 & 19&	  6& 23&	37.1\\
		     &	2 &      65\,025 & 28&	 8& 25&   35.8\\
		     &	3 &       16\,129 & 36&  10& 12&29.1\\
		     &	4 &        3\,969 & 41&	 10& 12&	31.0\\
		     &  5 &             961 & 29&  12& 14&	 24.1\\
		\bottomrule
	\end{tabular}%
	\label{tab:3}%
\end{table}%

\section{Comments and Perspectives}\label{sec:NeMODiscuss}
In this paper we analyzed a Newton-type multilevel optimization (NeMO) algorithm. We argued that the appropriate convergence rate for this multilevel algorithm should be composite i.e. it should have both a linear and quadratic component. We then studied the linear component of the composite rate, and we showed how the hierarchical structure of the model could be used to improve it. To our knowledge, this is the first time a connection between the hierarchal structure of the model and the rate of convergence of a multilevel optimization algorithm has been made. The results presented in this paper can be generalized and refined. The local composite rate of convergence when solving PDE-based optimization can be extended to cases beyond one-dimensional problems or uniform discretization. These extensions would require more careful analysis, but the general approach presented in Section \ref{sec:PDE} can be applied. 

\section*{Acknowledgements}
The second author was supported by FMJH/PGMO Project No2017-0088 ``Multi-level Methods in Constrained Optimization''. The third author was funded by Engineering \& Physical Sciences Research Council grant number EP/M028240/1. The third author was also funded in part by JPMorgan Chase \& Co. Any views or opinions expressed herein are solely those of the authors listed, and may differ from the views and opinions expressed by JPMorgan Chase \& Co. or its affiliates. This material is not a product of the Research Department of J.P. Morgan Securities LLC. This material does not constitute a solicitation or offer in any jurisdiction.

\bibliography{bibliography}

\begin{thebibliography}{10}

\bibitem{Beck2013}
A.~Beck and L.~Tetruashvili.
\newblock On the convergence of block coordinate descent type methods.
\newblock {\em SIAM Journal on Optimization}, 23(4):2037--2060, 2013.

\bibitem{Boyd:2004:CO:993483}
S.~Boyd and L.~Vandenberghe.
\newblock {\em Convex Optimization}.
\newblock Cambridge University Press, New York, NY, USA, 2004.

\bibitem{Briggs2000}
W.~L. Briggs, V.~E. Henson, and S.~F. McCormick.
\newblock {\em A Multigrid Tutorial}.
\newblock SIAM, 2000.

\bibitem{cal2019highorder}
H.~Calandra, S.~Gratton, E.~Riccietti, and X.~Vasseur.
\newblock On high-order multilevel optimization strategies, 2019.

\bibitem{cal2019approximation}
H.~Calandra, S.~Gratton, E.~Riccietti, and X.~Vasseur.
\newblock On the approximation of the solution of partial differential
  equations by artificial neural networks trained by a multilevel
  levenberg-marquardt method, 2019.

\bibitem{MR3742696}
Juan~S. Campos and Panos Parpas.
\newblock A multigrid approach to {SDP} relaxations of sparse polynomial
  optimization problems.
\newblock {\em SIAM J. Optim.}, 28(1):1--29, 2018.

\bibitem{Chan1994}
T.~F. Chan and B.~F. Smith.
\newblock Domain decomposition and multigrid algorithms for elliptic problems
  on unstructured meshes.
\newblock In {\em Domain decomposition methods in scientific and engineering
  computing ({U}niversity {P}ark, {PA}, 1993)}, volume 180 of {\em Contemp.
  Math.}, pages 175--189. Amer. Math. Soc., Providence, RI, 1994.

\bibitem{Erdogdu2015}
M.~A. Erdogdu and A.~Montanari.
\newblock Convergence rates of sub-sampled newton methods.
\newblock In {\em Advances in Neural Information Processing Systems 28: Annual
  Conference on Neural Information Processing Systems 2015, December 7-12,
  2015, Montreal, Quebec, Canada}, pages 3052--3060, 2015.

\bibitem{Gratton2010}
S.~Gratton, M.~Mouffe, A.~Sartenaer, P.~L. Toint, and D.~Tomanos.
\newblock Numerical experience with a recursive trust-region method for
  multilevel nonlinear bound-constrained optimization.
\newblock {\em Optimization Methods and Software}, 25(3):359--386, 2010.

\bibitem{Gratton2008}
S.~Gratton, A.~Sartenaer, and P.~L. Toint.
\newblock Recursive trust-region methods for multiscale nonlinear optimization.
\newblock {\em SIAM Journal on Optimization}, 19:414--444, 2008.

\bibitem{Hackbusch1985}
W.~Hackbusch.
\newblock {\em Multigrid methods and applications}, volume~4 of {\em Springer
  Series in Computational Mathematics}.
\newblock Springer-Verlag, Berlin, 1985.

\bibitem{Hackbusch2003}
W.~Hackbusch.
\newblock {\em Multi-Grid Methods and Applications}.
\newblock Springer, 2003.

\bibitem{Han2017}
J.~Han, Y.~Yang, and H.~Bi.
\newblock A new multigrid finite element method for the transmission eigenvalue
  problems.
\newblock {\em Applied Mathematics and Computation}, 292:96--106, 2017.

\bibitem{Ho2014}
C.~P. Ho and P.~Parpas.
\newblock Singularly perturbed {M}arkov decision processes: a multiresolution
  algorithm.
\newblock {\em SIAM Journal on Control and Optimization}, 52(6):3854--3886,
  2014.

\bibitem{MR3572365}
V.~Hovhannisyan, P.~Parpas, and S.~Zafeiriou.
\newblock M{AGMA}: multilevel accelerated gradient mirror descent algorithm for
  large-scale convex composite minimization.
\newblock {\em SIAM J. Imaging Sci.}, 9(4):1829--1857, 2016.

\bibitem{Kocvara2016}
M.~Ko{\v{c}}vara and S.~Mohammed.
\newblock A first-order multigrid method for bound-constrained convex
  optimization.
\newblock {\em Optimization Methods and Software}, 31(3):622--644, 2016.

\bibitem{Lewis2005}
R.~M. Lewis and S.~G. Nash.
\newblock Model problems for the multigrid optimization of systems governed by
  differential equations.
\newblock {\em SIAM Journal on Scientific Computing}, 26(6):1811--1837
  (electronic), 2005.

\bibitem{Lewis2013}
R.~M. Lewis and S.~G. Nash.
\newblock Using inexact gradients in a multilevel optimization algorithm.
\newblock {\em Computational Optimization and Applications}, 56(1):39--61,
  2013.

\bibitem{Nash2000}
S.~G. Nash.
\newblock A multigrid approach to discretized optimization problems.
\newblock {\em Optimization Methods and Software}, 14(1-2):99--116, 2000.
\newblock International Conference on Nonlinear Programming and Variational
  Inequalities (Hong Kong, 1998).

\bibitem{Nash2014}
S.~G. Nash.
\newblock Properties of a class of multilevel optimization algorithms for
  equality-constrained problems.
\newblock {\em Optimization Methods and Software}, 29(1):137--159, 2014.

\bibitem{MR3716579}
Panos Parpas.
\newblock A multilevel proximal gradient algorithm for a class of composite
  optimization problems.
\newblock {\em SIAM J. Sci. Comput.}, 39(5):S681--S701, 2017.

\bibitem{Press1996}
W.~H. Press, S.~A. Teukolsky, W.~T. Vetterling, and B.~P. Flannery.
\newblock {\em Numerical recipes: the art of scientific computing. {C}ode
  {CD}-{ROM} v 2.06 with {W}indows, {DOS}, or {M}acintosh single-screen
  license}.
\newblock Cambridge University Press, Cambridge, 1996.

\bibitem{Strang2007}
G.~Strang.
\newblock {\em Computational science and engineering}.
\newblock Wellesley-Cambridge Press, Wellesley, MA, 2007.

\bibitem{Trottenberg2001}
U.~Trottenberg, C.~W. Oosterlee, and A.~Sch{\"u}ller.
\newblock {\em Multigrid}.
\newblock Academic Press, Inc., San Diego, CA, 2001.
\newblock With contributions by A. Brandt, P. Oswald and K. St{\"u}ben.

\bibitem{Wen2009}
Z.~Wen and D.~Goldfarb.
\newblock A line search multigrid method for large-scale nonlinear
  optimization.
\newblock {\em SIAM Journal on Optimization}, 20(3):1478--1503, 2009.

\bibitem{Wesseling1992}
P.~Wesseling.
\newblock {\em An introduction to multigrid methods}.
\newblock Pure and Applied Mathematics (New York). John Wiley \& Sons, Ltd.,
  Chichester, 1992.

\end{thebibliography}
\bibliographystyle{plain}

\end{document}